\documentclass[letterpaper,reqno,11pt,oneside]{amsart}

\usepackage{graphicx}
\usepackage[totalheight=9.6in,totalwidth=6.15in,marginparwidth=1in,centering]{geometry}
\usepackage{colonequals}
\usepackage{amssymb,amsmath,latexsym,amsthm,amsfonts,dsfont,mathrsfs}
\usepackage[normalem]{ulem}
\usepackage{caption}
\usepackage{enumitem}
\usepackage{subcaption}
\usepackage{mathtools}
\usepackage[dvipsnames]{xcolor}
\usepackage{centernot}
\usepackage{xifthen}
\usepackage{xstring}
\usepackage{xparse}
\usepackage{etex}
\usepackage[symbol]{footmisc}
\usepackage{tikz}
\usepackage[extension=pdf,bookmarksopen]{hyperref}
\RequirePackage{cleveref}
\usepackage[style=alphabetic,sorting=nyt,natbib=true,maxnames=99,isbn=false,doi=false,url=false,firstinits=true,hyperref=auto,arxiv=abs,backend=bibtex]{biblatex}
\AtEveryBibitem{%
  \clearlist{language}}
\DeclareFieldFormat[article,inbook,incollection,inproceedings,patent,thesis,unpublished]{title}{#1\isdot}
\renewbibmacro{in:}{%
  \ifentrytype{article}{}{%
    \printtext{\bibstring{in}\intitlepunct}}}
\defbibheading{apa}[\refname]{\section*{#1}}
\DeclareFieldFormat{sentencecase}{\MakeSentenceCase{#1}} \renewbibmacro*{title}{%
  \ifthenelse{\iffieldundef{title}\AND\iffieldundef{subtitle}}{}
  {\ifthenelse{\ifentrytype{article}\OR\ifentrytype{inbook}%
      \OR\ifentrytype{incollection}\OR\ifentrytype{inproceedings}%
      \OR\ifentrytype{inreference}} {\printtext[title]{%
        \printfield[sentencecase]{title}%
        \setunit{\subtitlepunct}%
        \printfield[sentencecase]{subtitle}}}%
    {\printtext[title]{%
        \printfield[titlecase]{title}%
        \setunit{\subtitlepunct}%
        \printfield[titlecase]{subtitle}}}%
    \newunit}%
  \printfield{titleaddon}}

\definecolor{darkblue}{rgb}{0.13,0.13,0.39}%
\hypersetup{colorlinks=true,urlcolor=darkblue,citecolor=darkblue,linkcolor=darkblue,pdftitle={Coexistence for a population model with forest fires},pdfauthor={L. Fredes, A. Linker, D. Remenik}}

\newcommand{\llr}{<\hspace{-4pt}<}

\DeclareDocumentCommand\brho{ m g }{\tilde{\rho}_{#1}^N\IfNoValueF {#2} {(#2)}}
\DeclareDocumentCommand\beeta{ m g }{\tilde{\eta}_{#1}^{N\IfNoValueF {#2} {,(#2)}}}
\DeclareDocumentCommand\ff{ g }{f_{\beta}\IfNoValueF {#1}{^{(#1)}}}
\DeclareDocumentCommand\pb{ g g }{\mathsf{u}\IfNoValueF{#1}{_{#1}}\IfNoValueF{#2}{^{#2}}}
\DeclareDocumentCommand\gg{ g }{g \IfNoValueF {#1}{_{\alp{#1}}}}
\DeclareDocumentCommand\GG{ g }{G \IfNoValueF {#1}{_{\alp{#1}}}}
\DeclareDocumentCommand\ll{ g }{l \IfNoValueF {#1}{_{#1}}}
\newcommand{\sfrac}[2] {\mbox{$\frac{#1}{#2}$}}

\newcommand{\We}{\overline{W}}
\DeclareDocumentCommand\ss{ g }{\Sigma \IfNoValueF {#1}{_\alp{#1}}}

\DeclareDocumentCommand\rrho{ m g }{\rho_{#1}^{N}\IfNoValueF {#2} {\tsm(#2)}}
\DeclareDocumentCommand\eeta{ m g }{\eta_{#1}^{N\IfNoValueF {#2} {,(#2)}}}
\DeclareDocumentCommand\bet{ g }{\IfNoValueTF {#1} {\vec{\beta}}{\beta(#1)}}
\DeclareDocumentCommand\alp{ g g }{\alpha\IfNoValueF {#2} {_{#2}}\IfNoValueF{#1}{\StrLen{#1}[\x] \ifnum\x>0 (#1) \fi}}
\DeclareDocumentCommand\pphi{ g }{\phi\IfNoValueF {#1}{_{#1}}}
\DeclareDocumentCommand\ff{ g }{f_{\beta}\IfNoValueF {#1}{^{(#1)}}}
\DeclareDocumentCommand\pp{ g g }{p\IfNoValueF{#1}{_{#1}}\IfNoValueF{#2}{^{#2}}}
\DeclareDocumentCommand\hh{ g g }{h \IfNoValueF {#1}{^{#1}}\IfNoValueF{#2} {_{#2}}}
\DeclareDocumentCommand\bh{ g g }{\bar{h} \IfNoValueF {#1}{^{#1}}\IfNoValueF{#2} {_{#2}}}
\newcommand{\uno}[1]{\mathds{1}_{\{#1\}}}
\newcommand{\Pe}{\mathbf{P}}
\newcommand{\Ee}{\mathbf{E}}
\DeclareDocumentCommand\tt{ g }{\tau_N\IfNoValueF {#1}{^{#1}}}
\DeclareDocumentCommand\dd{ m g }{d_{#1}^{N\IfNoValueF {#2} {,(#2)}}}
\DeclareDocumentCommand\eetat{ m g }{\tilde{\eta}_{#1}^{N\IfNoValueF {#2} {,(#2)}}}
\DeclareDocumentCommand\eetac{ m g }{\hat{\eta}_{#1}^{N\IfNoValueF {#2} {,(#2)}}}
\newcommand{\ke}{\kappa_\varepsilon}
\DeclareDocumentCommand\qq{ g g }{q\IfNoValueF{#1}{_{#1}}\IfNoValueF{#2}{^{#2}}}
\DeclareDocumentCommand\pbhi{ g }{\varphi\IfNoValueF {#1}{^{#1}}}

\newcommand{\betac}{2\log 2}
\newcommand{\etac}{\sigma}
\newcommand{\MM}{\textup{MM} }
\newcommand{\MMM}{\textup{MMM} }
\newcommand{\E}{\mathbb{E}}
\newcommand{\Cr}{\mathcal{C}_r}

\newcommand{\V}{\mathbb{V}}

\newcommand{\N}{\mathbb{N}}
\newcommand{\R}{\mathbb{R}}
\newcommand{\Z}{\mathbb{Z}}
\newcommand{\ualpha}{\underline{\alpha}}
\renewcommand{\P}{\mathbb{P}}

\newcommand{\A}{{\sf A}}

\newcommand{\B}{{\sf B}}

\newcommand{\GN}{{G_N}}

\newcommand{\DS}{{\sf DS}}
\newcommand{\C}{{\sf C}}

\newcommand{\D}{{\sf D}}
\newcommand{\TT}{{\mathcal{T}}}

\newcommand{\ts}{\hspace{0.1em}}
\newcommand{\tts}{\hspace{0.05em}}
\newcommand{\tsm}{\hspace{-0.1em}}

\newtheorem{theo}{Theorem}[section]
\newtheorem{cor}[theo]{Corollary}
\newtheorem{lemma}[theo]{Lemma}
\newtheorem{prop}[theo]{Proposition}

\theoremstyle{definition}
\newtheorem{definition}[theo]{Definition}

\newtheorem{remark}[theo]{Remark}

\newcommand{\W}{{\mathcal{W}}}

\numberwithin{equation}{section}

\begin{document}

\title[Coexistence for a population model with forest fires]{Coexistence for a population model with forest fire epidemics}

\author{Luis Fredes}
\address[L.~Fredes]{LaBRI, University of Bordeaux}
\email{luis-maximiliano.fredes-carrasco@u-bordeaux.fr}

\author{Amitai Linker}
\address[A.~Linker]{Departamento de Matem\'aticas, Facultad de Ciencias Exactas, Universidad Andr\'es Bello, Santiago, Chile} 
\email{amitai.linker@unab.cl}

\author{Daniel Remenik} \address[D.~Remenik]{
  Departamento de Ingenier\'ia Matem\'atica and Centro de Modelamiento Matem\'atico (IRL-CNRS 2807)\\
  Universidad de Chile\\
  Av. Beauchef 851, Torre Norte, Piso 5\\
  Santiago\\
  Chile} \email{dremenik@dim.uchile.cl}

\date{March 14, 2022.}

\begin{abstract}
We investigate the effect on survival and coexistence of introducing forest fire epidemics to a certain two-species competition model. The model is an extension of the one introduced by \citet{DYR09}, who studied a discrete time particle system running on a random 3-regular graph where occupied sites grow until they become sufficiently dense so that an epidemic wipes out large clusters. In our extension we let two species affected by independent epidemics compete for space, and we allow the epidemic to attack not only giant clusters, but also clusters of smaller order. Our main results show that, for the two-type model, there are explicit parameter regions where either one species dominates or there is coexistence; this contrasts with the behavior of the model without epidemics, where the fitter species always dominates. We also discuss the survival and extinction regimes for the model with a single species. In both cases we prove convergence to explicit dynamical systems; simulations suggest that their orbits present chaotic behavior.
\end{abstract}

\maketitle

\section{Introduction and main results}\label{introduccion}

In the mathematical biology literature, resource competition between $n$ species is widely modeled through Lotka-Volterra type ODEs of the form
\[\textstyle\frac{dx_i(t)}{dt}\;=\;x_i(t)\tsm\left(a_i-\sum_{j=1}^nb_{ij}x_{j}(t)\right),\qquad i=1,\dotsc,n,\]
or suitable difference equation versions of them if time is taken to be discrete, where $x_i\in[0,1]$ represents the density of the $i$-th species and the $a_i$'s and $b_{ij}$'s are parameters. The term inside the parentheses determines the effect of inter-specific and intra-specific competition, and has the advantage of being simple enough for an easy interpretation of its coefficients while, at the same time, allowing the system to exhibit a rich asymptotic behavior, including fixed points, limit cycles and attractors. However, despite its ubiquitousness, the classical model seems inadequate to explain diverse and complex ecosystems, as conditions for stability become more restrictive for larger values of $n$; the same seems to be true regarding conditions for coexistence (see e.g. \cite{Hofbauer1987, intraspecific}), implying that, unless the parameters have been finely tuned, most species will be driven to extinction as a result of competition.
	
Even though it has been argued that natural selection alone may be able to tune the relevant parameters to yield a coexistence regime \cite{naturalsel1}, a considerable amount of effort has been directed towards extending models such as Lotka-Volterra in ways that promote biodiversity, for example through the addition of predators \cite{MIMURA,Hofbauer,schreiber}, of random fluctuations in the environment \cite{ZHU,Mao} and of diseases \cite{Holt,saenz}. Another way of extending the model is based on questioning the linear form of the inter-specific and intra-specific competition terms; indeed, for large population densities the intra-specific competition of a species has an increasingly important nonlinear component, known as the {\sl crowding effect}, which is overlooked in the original equations. The crowding effect is capable of effectively outbalancing the inter-specific competition effect for a significantly larger set of parameters, permiting coexistence even when $n$ is large \cite{Hartley,sevenster,Gavina}.

One important source for the crowding effect is the fact that at high population densities the connectedness between individuals tends to be high,  making it easier for an infectious disease to spread through the population and giving rise to epidemic outbreaks.
To the best of our knowledge, the effect that this phenomenon may have on coexistence has not been explored in the setting of competing spatial population models. This provides the main motivation for our paper.

\subsection{The multi-type moth model on a random 3-regular graph}

The model which we will study is a multi-type version of a particle system introduced by \citet{DYR09}.
Their model is inspired by the gypsy moth, whose populations grow until they become sufficiently dense for the nuclear polyhedrosis virus, which strikes at larval stage and spreads between nearby hosts, to reduce them to a low level; we will refer to it as the {\sl moth model (MM)}.
The MM is a discrete time particle system which alternates between a growth stage akin to a discrete time contact process and a forest fire stage where an epidemic randomly destroys entire clusters of occupied sites. Forest fire models, which were first introduced in \cite{DrosselSchwabl}, have received much interest as a prime example of a system showing self-organized criticality, see e.g. \cite{rath}, but this is not the focus of our paper.
\cite{DYR09} was devoted mostly to the study of the evolution of the density of occupied sites in the limit as the size of the system goes to infinity; its main result showed that the system converges to a discrete time dynamical system which, as a result of the forest fire epidemic mechanism, presents chaotic behavior.

The extension of the MM which we will be interested in, and which we call the {\sl multi-type moth model (MMM)}, is defined as follows.
Let $(\GN)_{N\geq1}$ be a random connected 3-regular graph of size $N$, i.e. a random graph chosen uniformly among all connected graphs with $N$ vertices, all of which have degree 3 (we condition on the graph being connected for simplicity, it is known that a random 3-regular graph is connected with probability tending to 1 as $N\to\infty$ \cite{janLuRu}).
Fix also $m\in\N$, which will be the number of species (we will be interested mainly in $m=1$ and $m=2$).
For each $N\in\N$ the {MMM} is a discrete time Markov chain $\big(\eeta{k}\big)_{k\geq0}$ taking values in $\{0,\dots,m\}^\GN$; each site $x\in\GN$ can be occupied by an individual of type $i \in \{ 1,...,m \}$ ($\eeta{k}(x)=i$) or be vacant ($\eeta{k}(x)=0$).
The process depends on two sets of parameters, $\beta=(\bet{1},\dotsc,\bet{m})\in \R^m_+$ and $\alpha_N=(\alp{1}{N},\dotsc,\alp{m}{N})\in[0,1]^{m}$.
The dynamics of the process at each time step is divided into two consecutive stages, {\sl growth} and {\sl epidemic}:

\vskip3pt

\noindent\underline{Growth:} An individual of type $i$ at site $x \in \GN$ sends a Poisson$[\bet{i}]$ number of descendants to sites chosen uniformly at random in $\GN$.
If a site receives more than one individual, the type of the site is chosen uniformly among the individuals it receives.
We will use the notation $\eta_{k+\frac{1}{2}}$ to refer to the configuration after the $k$-th growth stage but before the subsequent epidemics.

\noindent\underline{Epidemic:} Each site $x$ occupied by an individual of type $i$ after the growth stage is attacked by an epidemic with probability $\alp{i}{N}$, independently across sites.
The individual at $x$ then dies along with its entire connected component of sites occupied by individuals of type $i$.
This happens independently for $i=1,\dotsc,m$.

\vskip3pt

The MMM can be defined naturally running on any sequence of (random or deterministic) graphs $\GN$.
In this paper we choose to work on random 3-regular graphs mostly because they look locally like a regular tree, which leads to explicit formulas for certain percolation probabilities which will appear in the epidemic stage.
Our results should hold for other choices of graphs which have this property, but for simplicity we will not pursue this here.
Likewise, it is possible to work with more general offspring distributions, as done in \cite{DYR09}, but we stick to Poisson in order to simplify the presentation and proofs.

Observe that the growth stage in our model is of mean-field type.
This is a simplifying assumption, but is not totally unrealistic: in terms of the one-year life cycle of the gypsy moth, one may think of the individuals as performing independent random walks in $\GN$ between each time step of the process (that is, during the moth stage coming from larvae surviving the epidemic), so that the population will have mixed by the time new individuals are born and then the growth stage will be, effectively, approximately mean-field.
One could generalize the model by sending particles born at $x\in\GN$ in the growth step to a site chosen uniformly from some given neighborhood $\mathcal{N}_N(x)$ of $x$.
We believe that most of our results remain true in the {\sl spread-out} case corresponding to $\mathcal{N}_N(x)=B(x,r_N)$ (the ball of radius $r_N$ around $x$ in the natural graph distance) with appropriate growth conditions on $r_N$, but it is not clear to us whether our arguments can be extended to that setting.

Note on the other hand that while the growth parameters $\bet{i}$ are fixed, we have allowed the epidemic parameters $\alp{i}{N}$ to depend on $N$.
For each species we are interested in two basic possibilities: either $\alp{i}{N}\to\alp{i}\in(0,1)$ for all $i$, or $\alp{i}{N}\longrightarrow0$ slower than logarithmically.
In the second case, which we will refer to as the {\sl weak epidemic regime}, a fixed site is hit by the epidemic with negligible probability, but it will typically be infected when it belongs to a macroscopic (giant) cluster of occupied sites, and in this case the infection will typically come from a site which is most at logarithmic distance (see Section \ref{sec:derivlim}).
In the first case, the {\sl strong epidemic regime}, and on top of infections coming from other sites in a connected cluster, each occupied site is hit by the epidemic with probability bounded away from $0$; as we will see, the behavior of the system as $N\to\infty$ is different in the two cases.
The condition on infections arriving typically from neighbors at most at a logarithmic distance, which comes from the decay condition we imposed on $\alp{i}{N}$, is technical; it will allow us to approximate neighborhoods in $\GN$ at relevant scales by a tree.
In principle one could consider weaker epidemic regimes, where $\alp{i}{N}\to0$ faster than logarithmically and infections typically come from far away neighbors, but this situation seems to go beyond the methods in our paper (in particular, it is not clear what the $N\to\infty$ limit of the evolution of the densities of occupied sites would be in this case).

In order to incorporate both regimes in the notation, we will assume throughout most of the paper that there are fixed parameters $\alp{1},\dotsc,\alp{m}\in[0,1)$ so that
\begin{equation}
\alp{i}{N}\longrightarrow\alp{i}\quad\text{and}\quad\alp{i}{N}\log(N)\longrightarrow\infty\quad\text{as }N\to\infty,\quad i=1,\dotsc,m\label{eq:alphaconv}
\end{equation}
(note that we exclude the trivial case $\alp{i}=1$; note also that the second condition is trivial if $\alp{i}>0$).
We remark that, while the MM studied in \cite{DYR09} corresponds to the $m=1$ case of our MMM, that paper worked only in the weak epidemic regime, so some of our results extend theirs even in the single-type case.
This extension, which is natural from the biological point of view as it incorporates into the model the effect of diseases with a fixed incidence rate, has a major impact on the system, see Sections \ref{sec:derivlim} and \ref{sec:phasediag}.

For later use we introduce the sequence $\big(\rrho{k}\big)_{k\geq0}$ of density vectors obtained from $\big(\eeta{k}\big)_{k\geq0}$, defined as
\begin{equation}\label{MMMdens}
\rrho{k}=(\rrho{k}{1},\dotsc,\rrho{k}{m})\qquad\text{with}\quad\rrho{k}{i}=\frac{1}{N}\sum_{x\in\GN}\uno{\eeta{k}(x)=i}.
\end{equation}

\subsection{Coexistence and domination for the two-type MMM}\label{sec:coexdom}

If one suppresses the epidemic stage then the MMM turns into a multi-type contact process, for which it is relatively easy to prove that the fittest species (i.e. the one with the largest growth parameter $\bet{i}$) will outcompete and drive to extinction all the other ones (this has been proved for the contact process in continuous time with other choices of $\GN$, see e.g. the result of \cite{N92}, and it would not be hard to extend to the current setting).
Our main result, which we state and prove in the case $m=2$, shows that the introduction of forest fire dynamics changes this picture: there are choices of parameters for which there is coexistence even when one species has a larger offspring parameter.
The intuition behind this is simple: if we introduce forest fire epidemics into the system then the fitter species, which achieves higher densities, will be more susceptible to the destruction of large occupied clusters, which will have the effect of periodically clearing space for the growth of the weaker species, giving it a chance to survive.

In order to state our result we need to explain first what we mean by coexistence.
Let
\[\tt{i}=\inf\!\big\{k\geq 1\!:\eeta{k}(x)\neq i~\forall x\in\GN \big\}=\inf\!\big\{k\geq 1\!:\rrho{k}{i} = 0 \big\}\]
denote the extinction time of type $i$, for $i=1,2$.
Note that the MMM is a Markov chain on a finite state space with the all-empty configuration as its unique absorbing state, which will be reached eventually starting from any initial condition, so it makes no sense to ask any of the species to survive for all times.
We follow instead the usual approach (see e.g. \cite{cox1989,durrett1988}) where one characterizes the different phases of the system in terms of the behavior of the extinction times as a function of the network size $N$.
Roughly, given a timescale $s_N$ such that $s_N/\log(N)\longrightarrow\infty$, we will say that:
\begin{itemize}[itemsep=3pt,leftmargin=15pt]
\item Species $i$ {\sl dominates} species $j$ if there is a $c>0$ so that $\tau^j_N\leq c\tts\log N$ and $\tau^i_N\geq s_N$ with probability tending to 1 as $N\to\infty$.
\item The two species {\sl coexist} if $\tau^1_N,\tau^2_N\geq s_N$ with probability tending to 1 as $N\to\infty$.
\end{itemize}

Define the {\sl fitness} of species $i$ as
\begin{equation}
\pphi{i}\;=\;(1-\alp{i})\bet{i},\label{eq:def-fitness}
\end{equation}
which corresponds to the effective birth rate of individuals after considering the probability that a newly born particle does not survive the epidemic stage due to an infection arising in its location.
We are only interested in the regime $\pphi{1},\pphi{2}>1$, since when $\pphi{i}\leq1$ species $i$ dies out even when ignoring the other species and epidemics coming from other sites.
For concreteness we will assume that type 2 is the fitter species.

\begin{theo}\label{theo:5intro}
  Consider the two-species MMM on a random 3-regular graph satisfying \eqref{eq:alphaconv} and $1<\pphi{1}<\pphi{2}$ and let $\ualpha_N=\min\{\alpha_N(1),\alpha_N(2)\}$.
  Then there are constants $c_1,c_2,c_1',c_2'>0$ such that the following holds: For any fixed $0<l_1<u_1<1$ and $0<l_2<u_2<1$ there is a $C>0$ such that
  \begin{equation}
  \P\big(\tt{2}\geq e^{c_1\ts\ualpha_N\tsm\log(N)}\big)\geq 1-Ce^{-c_2\ts\ualpha_N\tsm\log(N)},\label{esperanzadom2}
  \end{equation}
  for all $N$ and any $\rrho{0}{1}\in[l_1,u_1]$, $\rrho{0}{2}\in[l_2,u_2]$ (that is, the stronger species survives), while:
  \begin{enumerate}[label=\textup{(\roman*)}]
    \item \textup{(Coexistence)}\enspace If $\pphi{2}$ is sufficiently large then there is a $\underline{\phi}\in(1,\pphi{2})$ depending only on $\pphi{2}$ and $\alp{2}$ such that for all $\pphi{1}\in(\underline{\phi},\pphi{2})$,
    \begin{equation}\label{esperanzacoex}
    \P\big(\tt{1}\geq e^{c_1\ts\ualpha_N\tsm\log(N)}\big)\geq 1-Ce^{-c_2\ts\ualpha_N\tsm\log(N)}.
    \end{equation}
    \item \textup{(Domination)}\enspace For any $\pphi{2}$ there is a $\overline{\phi}\in(1,\pphi{2})$ and depending only on $\pphi{2}$ and $\alp{2}$ such that if $\pphi{1}\in(1,\overline{\phi}$),
	\begin{equation}
	\P(\tt{1}\leq c_1'\log N)\geq1-\log(N) e^{-c_2'\tts\ualpha_N\tsm\log(N)}\label{esperanzadom1}\end{equation}
  \end{enumerate}
\end{theo}

A couple of remarks are in order.

\begin{remark}\label{rem:dichotomy}
\leavevmode
\begin{enumerate}[label=(\roman*),itemsep=3pt,leftmargin=25pt]
\item In order for the result to provide a dichotomy between domination and survival, and fit the notions introduced above, one needs to have $\ualpha_N\log(N)/\log(\log(N))\longrightarrow\infty$ as $N\to\infty$. Note that this assumption also ensures that the right hand side of \eqref{esperanzadom1} goes to $1$.
\item Under the assumption $\ualpha_N\log(N)/\log(\log(N))\longrightarrow\infty$ one can prove that all the factors $\ualpha_N\log(N)$ appearing in the exponents in \eqref{esperanzadom2}--\eqref{esperanzadom1} can be replaced by $\ualpha_N\log(N)\vee\log(N)^{1/2}$, thus strengthening the dichotomy whenever $\log(\log(N))/\log(N)\llr\ualpha_N\llr\log(N)^{-1/2}$. See \Cref{rem:lalala} after the proof of \Cref{teo1paso}.
\item The timescale difference which we obtain is probably not optimal, but in any case it is quite strong: for example, if we take $\alp{i}{N}\longrightarrow\alp{i}\in(0,1)$ for each $i$ then the dichotomy for species 1 corresponds roughly to the difference between dying out in time $\log(N)$ and surviving for a time of order $N^c$ for some $c>0$.
\end{enumerate}
\end{remark}

\begin{remark}\label{rem:cdref}
Theorem \ref{theo:5intro} is a slightly simplified and condensed version of the results we will prove in later sections, which together provide finer information about the phase diagram of the process and of the dynamical system which describes it in the $N\to\infty$ limit, see Theorems \ref{theo:5} and \ref{extcoex}.
Those results imply in particular (see the discussion following the statement of \Cref{theo:5}) that, under the assumptions of \Cref{theo:5intro}:
\begin{enumerate}[label=(\roman*),itemsep=3pt,leftmargin=25pt]
\item There exist $\pphi{1}<\pphi{1}'<\pphi{2}$ such that type 2 dominates over type 1 in the $\MMM$ associated to $(\pphi{1},\pphi{2})$, while there is coexistence in the MMM associated to $(\pphi{1}',\pphi{2})$. This can be achieved, moreover, when $\alp{1}=\alp{2}=0$.
\item For any small $\gamma>0$ we can choose $\pphi{1}$ and $\pphi{2}$ large but with relative fitness $\frac{\pphi{1}}{\pphi{2}}=\gamma$ such that both species coexist. 
\item In particular, given any small $\gamma>0$ one can choose two different sets of parameters with the same relative fitness $\gamma$ so that in one case type 1 is driven to extinction while in the other case there is coexistence.
Hence, and in contrast to models such as the multi-type contact process, relative fitness by itself is not enough to predict the qualitative behavior of the system.
\end{enumerate}

\noindent \Cref{Ej40} contains a sketch of the regions of the phase diagram of the process which have been probed in \Cref{theo:5}, which in particular makes these four facts apparent.
\end{remark}

Note that in our model we are assuming that epidemics affect each species independently.
This is natural when considering epidemics lacking cross-species transmission due to genetic distance, but is not a very realistic assumption if one thinks about the competition of different species of trees and takes the forest fire metaphor literally. It seems, nevertheless, that this assumption is important for coexistence to arise in our setting. This qualitative difference between epidemics with and without cross-species transmission is somewhat similar to the one found in the literature for predators, where the addition of a ``specialist" predator to Lotka-Volterra systems can be more effective in promoting coexistence than the addition of a ``generalist" one (see \cite{schreiber}).

A related model was studied by \citet{chan2006}, who proved coexistence for the two-type, continuous time contact processes in $\Z^2$ with the addition of a different type of forest fires, which act by killing all individuals (regardless of their type, and regardless of whether they are connected) within blocks of a certain size. 
They showed that if the weaker competitor has a larger dispersal range then it is possible for the two species to coexist in the model with forest fires; this contrasts with Neuhauser's result \cite{N92} for the model without forest fires for which such coexistence is impossible.
Our context is different, since we work on a random graph with forest fires which travel only along neighbors of the same type and which have an unbounded range, and since all species use the same (mean-field) dispersal neigborhoods. 
The techniques we use are also different, and the results we obtain are of a slightly different nature.
But the motivation is similar, and our results complement nicely with theirs.

The strategy we will use to prove Theorem \ref{theo:5intro} proceeds in three steps which can be described roughly as follows: first we approximate the evolution of the densities of sites occupied by each type as $N\to\infty$ by an explicit deterministic dynamical system, then we study the phase diagram of this dynamical system to find regions for coexistence and survival, and finally we argue that on those regions the behavior of our process tracks that of the limiting dynamical system.
The main challenge in implementing this strategy comes from the slow convergence of the empirical densities to the limiting dynamical system.
This is intrinsic in the very nature of our model: as we will explain in Section \ref{sec:phasediag}, and just as in the single-type case, due to the forest fire epidemics the two-type dynamical system presents a very complicated behavior which, from simulations, appears to be chaotic; this makes it hard to obtain a fine control on the distance between the finite system and its limit, for which it is essentially impossible to predict its evolution.
As a consequence, in the coexistence regime we are not able to show that the extinction times of both species grow exponentially in $N$ even in the case of mean-field growth, as one would expect.

Our proof of coexistence relies on showing that a certain quantity, $\pphi{1}\etac$, is larger than 1, where $\etac$ is defined in \eqref{eq:defeta} and represents the average competition effect that the strong species has on the weaker one when the latter is close to extinction. 
A similar argument could be used to show that if an analogous quantity $\pphi{1}\bar{\etac}$ is smaller than one (with $\bar{\etac}$ defined by changing inf's by sup's in \eqref{eq:defpbphi} and \eqref{eq:defeta}), then the weaker species decreases to extinction as soon as it reaches sufficiently small densities.
It is not unreasonable to conjecture that in fact the condition $\pphi{1}\bar{\etac}<1$ implies domination, and furthermore that $\etac$ and $\bar{\etac}$ should coincide, which would characterize a complete dichotomy for the qualitative behavior of the system, but pursuing this is outside the scope of this paper.

\section{The limiting dynamical system}\label{sec:limds}

Throughout the paper we will use the notation $\DS(h)$ to denote the dynamical system $\big(h^n(p)\big)_{n\geq0}$ defined from the iterates $h^n$ of a given map $h\!:\R^m\longrightarrow\R^m$.

\subsection{Derivation of the limit}\label{sec:derivlim}

The starting point of our arguments is an approximation of the evolution of the MMM densities by a deterministic dynamical system.
We begin by explaining where this limit comes from.
Since it makes no difference, we work here in the case of general $m\geq1$.

Recall that the epidemic parameters satisfy $\alp{i}{N}\longrightarrow\alp{i}\in[0,1)$ as $N\to\infty$.
Since the MMM dynamics is defined in two stages, it is natural to look for maps $f_{\beta},g_{\alpha}\!:\R^m\longrightarrow\R^m$ describing respectively the limiting densities after the growth and epidemic stages and then expect the limiting dynamical system to be given by $\DS(g_{\alpha}\circ f_{\beta})$.

Recalling the Poisson assumption on the offspring distribution, and since in the process we let each site choose its type uniformly at random from the particles it receives, a simple computation shows that the expected density of sites occupied by type $i$ after the growth stage is given by
\begin{equation}
\ff{i}(p) \coloneqq \left( 1-e^{-\sum_{i=1}^m\bet{i}\pp{i}} \right)\frac{\bet{i}\pp{i}}{ \sum_{i=1}^m\bet{i}\pp{i} }.\label{eq:ffMMM}
\end{equation}
This is our candidate function for the growth part.
The function $g_{\alpha}$, on the other hand, will depend on our particular choice of a random 3-regular connected graph for $G_N$.
In this case the graph looks locally like a 3-regular tree, so in order to guess a candidate for $g_{\alpha}$ we can pretend that the epidemic stage acts on the infinite 3-tree $\TT$. 
Let us also assume for a moment that $m=1$.
We need to analyze the effect of the epidemic when attacking a configuration of particles distributed as independent (thanks to the mean-field assumption) site percolation on $\TT$ with a given density $q$ (whose distribution, i.e. a product measure on $\{0,1\}^\TT$ where each vertex is occupied with probability $q$, we denote as $\Pe_{q}$).
Note that if $\Cr$ denotes the connected component of occupied sites containing $r$ then, conditionally on $\Cr$, the probability that $r$ survives is given by $(1-\alp{}{N})^{|\Cr|}\uno{|\Cr|>0}$.

As a consequence, we should expect the limiting probability that a given site is occupied, after the epidemic stage attacks a configuration with a fraction $q$ of occupied sites, to be given by
\[g_\alpha(q)\coloneqq\Pe_q(r \text{ is occupied},\,r \text{ survives the epidemic})=\Ee_q((1-\alpha)^{|\Cr|}\uno{|\Cr|>0})\]
(here $r$ is any vertex of $\TT$).
The right hand side can be computed explicitly:

\begin{prop}
	\label{prop1}
	For any $q\in[0, 1]$ and $\alpha\in(0,1)$,
  \begin{equation}\label{eq:galph}
  g_\alpha(q)= \tfrac{\left( 1-\sqrt{1-4(1-\alpha)q(1-q)} \right)^3}{8 (1-\alpha)^2q^2},
  \end{equation}
  while $g_0(q)=\lim_{\alpha\to0^+}g_\alpha(q)$, which equals $q^{-2}(1-q)^3$ for $q\geq1/2$ and $q$ for $q<1/2$.
\end{prop}

The formula for $g_0$ coincides with the function appearing in \cite{DYR09}; the fact that $g_0(q)=q$ for $q<1/2$ reflects that in the weak epidemic regime $\alpha_N\to0$ the epidemic can only hit a giant cluster, which for site percolation on the 3-regular tree is seen only for $q\geq1/2$.
In contrast, when $\alpha>0$ the epidemic also attacks small clusters and the density of the population does not have to be above the critical percolation parameter of the network for it to kick in, so we observe its effects at all times.

Going back to the general case $m\geq1$, since the epidemic attacks each species independently and without cross-transmission, we deduce that the density of sites occupied by type $i$ after the epidemic stage acts on a population with initial densities  $q\in[0,1]^m$ should be given by
\begin{equation}
g_{\alpha}^{(i)}(q)=g_{\alp{i}}(q_i).
\end{equation}
In view of the above computations we define the candidate limiting dynamical system as $\DS(h)$ where, given $p\in[0,1]^m$, $p_1+p_2+\dots+p_m\leq 1$, $h(p)=\big(\hh{}{1}(p),\dotsc,\hh{}{m}(p)\big)$ is defined as
\begin{equation}
\hh{}{i}(p)=g_{\alp{i}}\circ \ff{i}(p)\label{eq:limh}
\end{equation}
(we omit the dependence of $h$ on the parameters for simplicity).

\subsection{Approximation result}

Recall the definition of the density process $\big(\rrho{k}\big)_{k\geq0}$ associated to the MMM.
A straightforward consequence of the following result (stated as Corollary \ref{cor4} below) is that the density process converges indeed to the dynamical system $\DS(h)$.
The result, however, goes much further, providing a quantitative estimate on the speed of convergence, which will be crucial in the proof of Theorem \ref{theo:5intro}.

\begin{theo}
	\label{teo1paso}
	Consider the MMM with $m$ types and assume that \eqref{eq:alphaconv} holds.
	Then given $\delta>0$ and $k\in\N$ there is a constant $C>0$, depending only on $\delta$ and $k$, such that for all $N\in\N$ and any initial condition $\eeta{0}$ we have (with $\ualpha_N=\min\{\alp{1}{N},\ldots,\alp{m}{N}\}$)
	\begin{equation}
	\label{enunc}
	\P\!\left(\big\|\rrho{k}-\hh{k}{}(\rrho{0})\big\|_\infty>\delta\right)\;\leq\;Ce^{-\ualpha_N\tsm\log_2(N)/5},
	\end{equation}
	where $\|x\|_\infty = \max_{i\in \{1,2,\dots,m\}} |x_i|$ for a vector $x\in\R^m$ ($\ell_\infty$ norm in $\R^m$).
\end{theo}

The bound on the right hand side is certainly not sharp but, as we explained in \Cref{rem:dichotomy}(i), it is strong enough for the purpose of deriving a dichotomy between domination and coexistence, as established in Theorem \ref{theo:5intro}.
That the bound gets better as $\ualpha_N$ gets larger is not surprising: the main contribution to the variability of the trajectory comes from the epidemic stage, which typically affects connected clusters with sizes of order $1/\ualpha_N$.
The main ingredient in the proof of this result is \Cref{convarbol}, which uses a comparison with a branching process to estimate the difference between $g$ and the expectation of the density obtained after the epidemic stage on a percolated 3-tree.

\begin{cor}
	\label{cor4}
	Suppose that \eqref{eq:alphaconv} holds and that $\rrho{0}$ converges to some $\pp$ such that $p_1+p_2+\dots+p_m\leq 1$, then  as $N\to\infty$, the density process $\big(\rrho{k}\big)_{k\geq 0}$ associated to the MMM converges in distribution (on compact time intervals) to the deterministic orbit, starting at $\pp$, of the dynamical system $\DS(h)$.
\end{cor}

In the case with $m=1$ and $\alp{1}=0$, this is Theorem 2 of \cite{DYR09}.

\subsection{Phase diagrams}\label{sec:phasediag}

Our goal here is to determine parameter regions for the two-type $\DS(h)$ where domination and coexistence hold.
In this context we say that (here $h_i^k$ denotes the $i$-th coordinate of the $k$-th iterate of $h$):
\begin{itemize}[itemsep=3pt,leftmargin=15pt]
\item Species $i$ {\sl dominates} species $j$ if $\liminf_{k\to\infty}h^k_i(\vec{p})>0$ while $\lim_{k\to\infty}h^k_j(\vec{p})=0$.
\item There is {\sl coexistence} if $\liminf_{k\to\infty}h^k_i(\vec{p})>0$ for $i=1,2$.
\end{itemize}
In order to investigate the behavior in the two-type case it is instructive to first review the behavior of the limiting dynamical system for single-type MM, for which a very complete picture is available.

\subsubsection{The one-type system and bifurcation cascades}
Consider the case $m=1$.
For simplicity, in this case we omit the subscripts from the parameters defining the process.
In the weak epidemic regime for the MM, $\alpha_N\longrightarrow0$ (which corresponds to $\alpha=0$ in $\DS(h)$), we are back in the case studied in \cite{DYR09}.
In that situation one has the following:
\begin{itemize}[itemsep=3pt,leftmargin=15pt]
  \item If $\beta\leq 1$ then for every $p\in[0,1]$ the sequence $\hh{k}(p)$ decreases to $0$ as $k\rightarrow\infty$, $0$ being the unique fixed point of $f_\beta$ (and $\hh$).
  \item The epidemic is only seen when the system attains densities larger than $1/2$. 
  Since the unique fixed point $p^*$ of $f_\beta$ is in $(0,1/2)$ for all $\beta\in(1,2\log(2)]$, for such $\beta$ the orbit of $h^k(p)$ eventually gets trapped inside the interval $[0,\frac{1}{2}]$, where there are no epidemic outbreaks ($h\equiv f_{\beta}$).
  Inside this interval, $h^k(p)$ converges to $p^*$.
  \item If $\beta>\betac$ then the orbit of $h^k(p)$ is trapped inside the interval $[h(\frac12),\frac{1}{2}]$.
  In this case the fixed point of $f_\beta$ is larger than $\frac12$, so the successive growth stages drive the density above this value, at which time the epidemic kicks in and forces a relatively large jump back to $[h(\frac12),\frac{1}{2}]$.
  In this case $\DS(h)$ is chaotic (see \cite[Thm. 1]{DYR09}).
\end{itemize}

Thus the case $\beta\leq 1$ corresponds to the {\sl extinction} regime (at least for the limiting dynamical system), while for all $\beta>1$ we have $\liminf_{k\to\infty}h^k(p)>0$ (for all $p\geq0$), which corresponds to {\sl survival}.
In \cite{DYR09} the authors also prove versions of these results (including the convergence to the corresponding dynamical system) for the process running on the discrete torus.

For the dynamical system $\DS(h)$ with general $\alpha\in[0,1]$ we have:

\begin{prop}
  \label{prop8}
  Let $\alpha\in[0,1]$ and $\beta>0$.
  \begin{enumerate}[label=\emph{(\roman*)}]
    \item \textup{(Extinction)}\enspace If $(1-\alpha)\beta\leq 1$ then $\lim_{k\to\infty}\hh{k}(p)=0$ for all $p\in [0,1]$.
    \item \textup{(Survival)}\enspace If $(1-\alpha)\beta>1$ then $\liminf_{k\to\infty}\hh{k}(p)>0$ for all $p \in (0,1)$.
  \end{enumerate}
\end{prop}

This result follows relatively easily from showing that, as a fixed point of $\DS(h)$, $0$ is attractive in case (i) and repulsive in case (ii), so we omit the proof.

The remaining question in the case of general $\alpha$ is to investigate the existence of a chaotic phase.
While a rigorous analysis appears to be much more difficult in this case due to the complicated algebraic structure of $h$, numerical simulations of the orbits of $\DS(h)$ suggest that the system presents {\sl bifurcation cascades}.
These are sequences of period doubling bifurcations that occur as the parameter $\beta$ is increased (for fixed $\alpha>0$), and which accumulate at a certain finite value of $\beta$ (the prototypical example of this behavior is the dynamical system defined by the quadratic map $x\longmapsto rx(1-x)$, which has a first period doubling bifurcation occurring at $r=3$ and then subsequent ones which continue up to $r\approx3.56$, where a chaotic regime arises; this phenomenon presents an intriguing form of universality \cite{FEI78,PCCT78}, see \cite{bifurcationCascades} for a good recent account). 
The bifurcation cascades appearing for $\alpha>0$ contrast with the behavior in the case $\alpha=0$, where the system proceeds directly from a stable fixed point to a chaotic phase, without passing through period-doubling bifurcations (see the discussion preceding \cite[Prop. 1.1]{DYR09} there); the parameter $\alpha$ has thus the effect of modulating the appearance of these cascades.
The left side of Figure \ref{figcascades} shows bifurcation diagrams for $\DS(h)$ which clearly suggest the occurence of this phenomenon in our system, while the right side shows a simulation of the evolution of the MM for finite $N$ and different values of $\beta$; note how some of the period doubling bifurcation behavior of the limiting system are still apparent in these simulations.

\begin{figure}[t]
  \begin{subfigure}{0.49\textwidth}
    \includegraphics[scale=0.4]{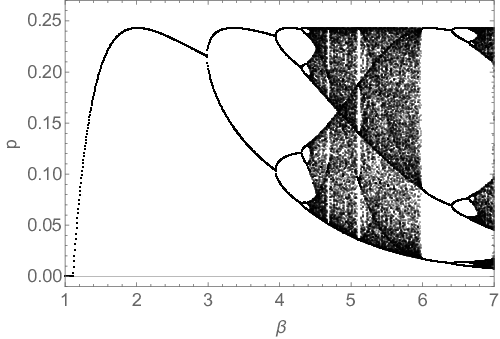}
  \end{subfigure}
  \begin{subfigure}{0.49\textwidth}
    \includegraphics[scale=0.4]{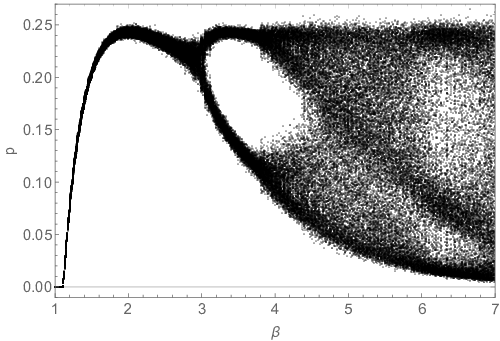}
  \end{subfigure}
\caption{{\small Left: Bifurcation diagram in $\beta$ for $\DS(h)$ with $\alpha=0.1$, showing the orbits of the system between iterations 900 and 1000 in the vertical direction for different values of $\beta$. Our simulations suggest that cascades appear for all $\alpha \in (0,1)$.\\[2pt]
\noindent Right: Simulation of the evolution of the MM for $\alpha=0.1$ and different values of $\beta$, from iteration 900 to 1000. Here $N\in \{20000,40000,100000\}$ (depending on $\beta$).}}
\label{figcascades}
\end{figure}

Figure \ref{f5} presents a schematic summary, partly based on simulations, of the behavior of the orbits of $\DS(h)$ as a function of $\alpha$ and $\beta$.

\begin{figure}[!h]
  \centering
  \includegraphics[scale=0.23]{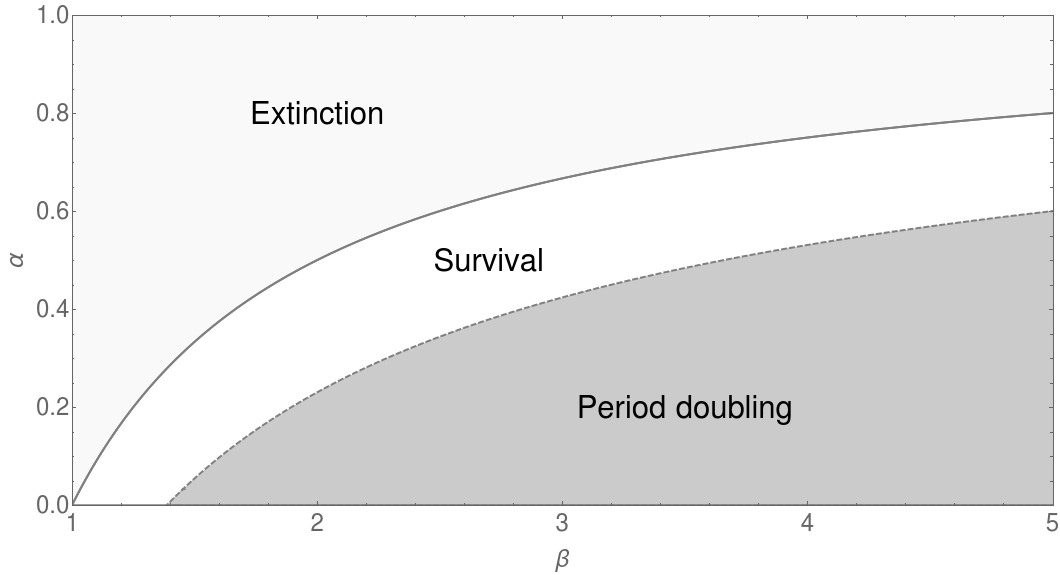}
  \caption{{\small Approximate phase diagram of $\DS(h)$. The transition between extinction and survival is justified by \Cref{prop8}, while the one governing the appearance of bifurcation cascades (dashed line) is based on simulations.}}
  \label{f5}
\end{figure}

\begin{remark}
The above discussion refers only to the behavior of the limiting dynamical system, and it is natural to wonder also about the dichotomy between extinction and survival at the level of the single-type particle system for finite $N$.
The phase diagram of the system in this case is much simpler than in the two-type setting of Theorem \ref{theo:5intro}, and one expects that if $(1-\alpha_N)\beta\longrightarrow\phi$ then the extinction time $\tt$ of the process should have a qualitatively different behavior in the cases $\phi<1$ and $\phi>1$.
In fact\footnote{See \url{https://arxiv.org/abs/1811.12468v3}, Sections 1.3 and 4.}, a simple comparison with a branching process shows that, for $\phi<1$ (the extinction phase) $\E(\tt)\leq C_1\log(N)$, while a separate, relatively simple argument, shows that for $\phi>1$ (the survival phase) $\E(\tt)\geq C_2N$ (for some fixed constants $C_1,C_2>0$).
We believe that in the survival phase the expected extinction time actually grows exponentially, i.e. that there are constants $c,C>0$ such that $\E(\tt)\geq C\tts e^{c N}$.
\end{remark}

\subsubsection{The two type dynamical system}

We come back now to the case $m=2$.
In this case a full description of the phase diagram as in \Cref{prop8} becomes extremely difficult to obtain due to the complicated explicit function $h$ arising from the competition between species.
In the following result we find instead some partial conditions which ensure either domination or coexistence.
In view of \Cref{prop8}, we will restrict the discussion to the case when the fitnesses of both species (defined in \eqref{eq:def-fitness}) satisfy $\pphi{i}>1$. 
For concreteness we will also assume that type 2 is fitter than type 1, i.e. $\pphi{2}>\pphi{1}$, and in order to ease notation, in everything that follows we denote, for a given initial condition $p\in[0,1]^2$ with $p_1+p_2\leq 1$ and any $i\in\{1,2\}$,
\[\pp{i}{k} = \hh{k}{i}(p).\]

\begin{theo}\label{teodinamico}
Consider the two-type dynamical system $\DS(h)$ with an arbitrary initial condition $p\in(0,1)^2$ with $p_1+p_2\leq 1$.
Then for any $1<\pphi{1}<\pphi{2}$
\[\liminf_{k\to\infty}\pp{2}{k}>0,\]
that is, the stronger species survives, while there are continuous functions $\mathcal{F}_1,\mathcal{F}_2:[0,1]\times\mathbb{R}^+\to\mathbb{R}$ such that:
\begin{enumerate}[label=\emph{(\roman*)}]
	\item \textup{(Coexistence)}~ If $\pphi{2}>2\log 2$ and $\pphi{1}>\mathcal{F}_1(\alp{2},\pphi{2})$, then 
  \[\liminf_{k\to\infty}\pp{1}{k}>0.\]
	\item \textup{(Domination)} ~If $\pphi{1}<\mathcal{F}_2(\alp{2},\pphi{2})$ then
  \[\lim_{k\to\infty}\pp{1}{k}=0.\]
\end{enumerate}
The functions $\mathcal{F}_1$ and $\mathcal{F}_2$ satisfy:
\begin{enumerate}[label=\emph{(\arabic*)}]
	\item For fixed $\alpha$, $\mathcal{F}_1(\alpha,\pphi)$ is increasing as a function of $\pphi$ and satisfies $\mathcal{F}_1(\alpha,\pphi)=\Theta(\sqrt{\pphi\log(\pphi)})$ for large $\pphi$. In particular, for large $\pphi$ we have $\mathcal{F}_1(\alpha,\pphi)<\pphi$.
	\item $\mathcal{F}_2(\alpha,\pphi)>1$ and for fixed $\alpha$, $\mathcal{F}_2(\alpha,\pphi)= 1+(1+o(1))\frac{\pphi}{2}e^{-\frac{3\pphi}{2}}$ for large $\pphi$. On the other hand, for fixed $\pphi$, $\mathcal{F}_2(\alpha,\pphi)$ is decreasing as a function of $\alpha$.
\end{enumerate}
\end{theo}

\begin{figure}[h]
	\centering
	\begin{tabular}{ll}
		\includegraphics[scale=0.3]{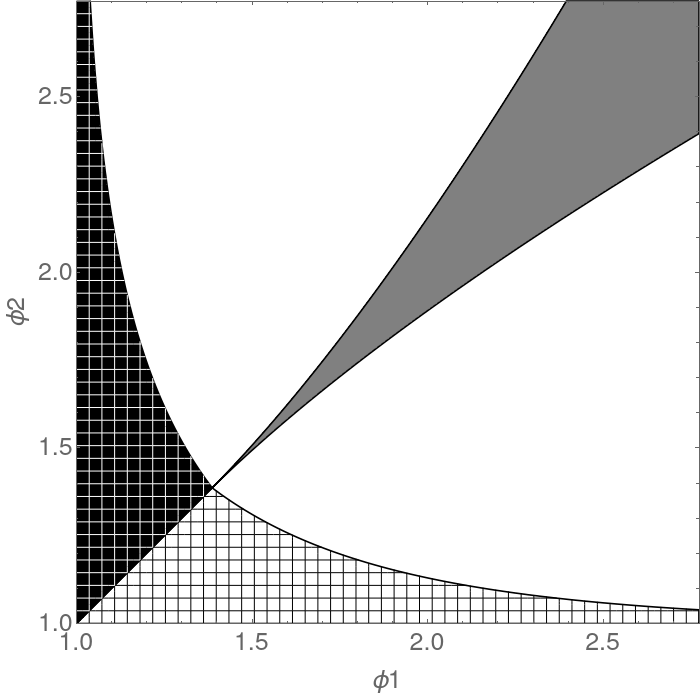}
		&
		\includegraphics[scale=0.3]{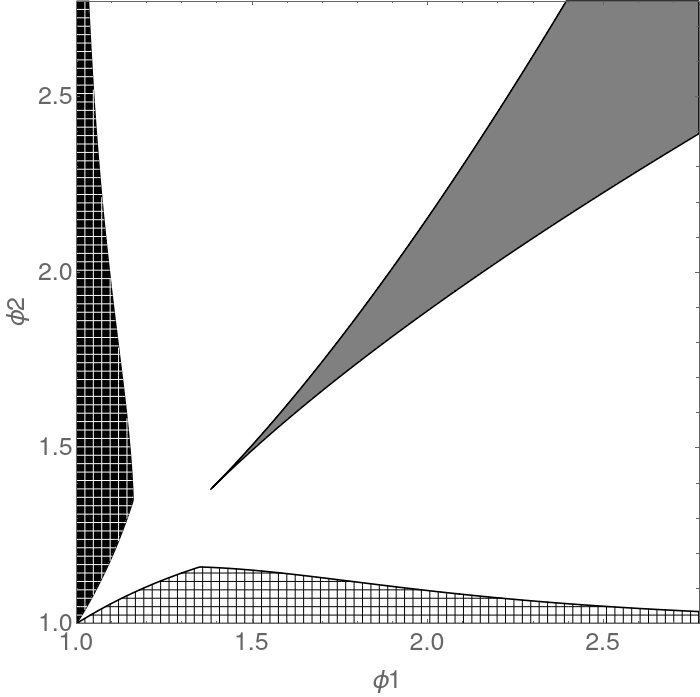}
	\end{tabular}
	\caption{{\small Summary of the domination and coexistence regimes for the MMM, for $\alp{1}=\alp{2}=0$ on the left and $\alp{1}=\alp{2}= 0.1$ on the right. 
			The white (resp. black) dashed regions represent the domination regime of type 1 over type 2 (resp. type 2 over type 1), and the solid gray regions correspond roughly to the coexistence regime (plotted based on their asymptotic behavior: as $\pphi{2}\rightarrow \infty$, $\pphi{1}$ grows as $\sqrt{\pphi{2}\log(\pphi{2})}$).
	}}
	\label{Ej40}
\end{figure}
We believe that the condition $\pphi{2}>2\log 2$ is not fundamental for coexistence and could be relaxed by carefully modifying our proofs.
The (rather complicated) definitions of the functions $\mathcal{F}_1$ and $\mathcal{F}_2$ are given in \eqref{eq:defcf1} and \eqref{eq:defcf2}, for which the properties described in the theorem can be proved analytically but whose numerical plots reveal additional features such as concavity of $\mathcal{F}_1$ and that $\mathcal{F}_2$ has a single critical point. An approximate phase diagram is given in Figure \ref{Ej40}, where it can be appreciated that as the $\alp{i}$'s increase the inequalities, the conditions become more restrictive and hence the regions given by the theorem shrink; this is a consequence of the chaotic behavior introduced by the epidemic stage, which reduces our control over the system, and which increases with the $\alp{i}$'s. Bifurcation diagrams corresponding to domination and coexistence regimes are shown in Figure \ref{Ej50}.

The intuition behind our coexistence result is the following. 
If the trajectory of the weaker type $1$ species remains close to zero, its effect on the trajectory of the type 2 species becomes negligible, meaning that type 2 evolves essentially as if it were alone, so that the evolution of $\pp{1}{k}$ can be approximated taking that of $\pp{2}{k}$ as given.
Condition $\pphi{1}>\mathcal{F}_1(\alp{2},\pphi{2})$ ensures that in this situation the type 1 species grows in average, thus moving away from low density values.
In the case of domination, the idea will be that starting from any initial condition the orbit of the dynamical system eventually gets stuck in a set $B$ where the condition $\pphi{1}<\mathcal{F}_2(\alp{2},\pphi{2})$ ensures that $\pp{1}{k}$ decays (exponentially fast) to 0.

\begin{figure}[h!]
  \centering
  \begin{tabular}{ll}
    \includegraphics[scale=0.4]{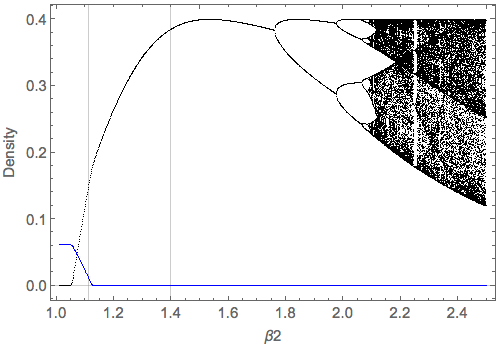}
    &
    \includegraphics[scale=0.4]{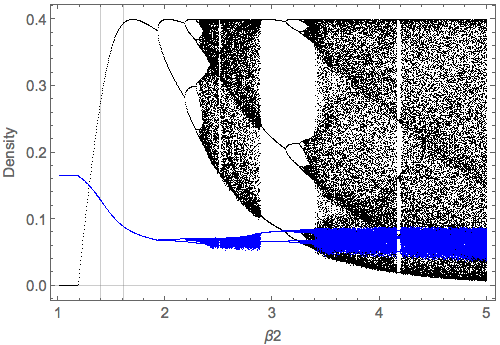}
  \end{tabular}
  \caption{{\small Bifurcation diagrams for type 1 (blue) and type 2 (black), with $\alp{1}=0.01$ and $\alp{2}=0.2$. On the left, with $\beta_1=1.99\log(2)$, type 1 goes from a stable fixed point to extinction as $\beta_2$ increases. On the right, with $\beta_1=2$, there is coexistence for large $\beta_2$; note how the chaotic behavior of the type 2 species is reflected on type 1 as well.}\label{Ej50}}
\end{figure}

\subsection{Connection with the particle system}
\label{conMMM}

We are finally ready to state the precise version of our main result (stated above as Theorem \ref{theo:5intro}), which extends the behavior derived in last section for the dynamical system $\DS(h)$ to the particle system.

\begin{theo}
  \label{theo:5} Let $\mathcal{F}_1$ and $\mathcal{F}_2$ be as in Theorem \ref{teodinamico}. For the two-species MMM on a random 3-regular graph, and under the assumptions of \Cref{theo:5intro}, we have: \textup{(i)} If $\pphi{2}>2\log 2$ and $\pphi{1}>\mathcal{F}_1(\alp{2},\pphi{2})$, then the coexistence statement \eqref{esperanzacoex}/\eqref{esperanzadom2} holds. \textup{(ii)} If $\pphi{1}<\mathcal{F}_2(\alp{2},\pphi{2})$, then the domination statement \eqref{esperanzadom1}/\eqref{esperanzadom2} holds.
\end{theo}

\Cref{theo:5intro} follows directly from this result when taking $\underline{\phi}=\mathcal{F}_1(\alp{2},\pphi{2})$ and $\overline{\phi}=\mathcal{F}_2(\alp{2},\pphi{2})$, since the former is smaller than $\pphi{2}$ for all sufficiently large $\pphi{2}$, and the latter is always larger than $1$. Observe that the properties of $\mathcal{F}_1$ and $\mathcal{F}_2$ give the behavior stated in \Cref{rem:cdref}. Indeed, for the first item we can fix a large $\pphi{2}$ and then take $\pphi{1}'$ close to $\pphi{2}$ so that there is coexistence for the pair $(\pphi{1}',\pphi{2})$, while at the same time taking $\pphi{1}$ close to $1$ so that there is domination for the pair $(\pphi{1},\pphi{2})$. For the second and third items of \Cref{rem:cdref} we can fix a large $\pphi{2}$ and $\pphi{1}$ close to $1$ so that the system exhibits domination while having relative fitness $\frac{\pphi{1}}{\pphi{2}}=\gamma$ as small as wanted. By taking $\pphi{2}'$ even larger we can choose $\pphi{1}'$ close to (but larger than) $\underline{\pphi}$, which is $\Theta(\sqrt{\pphi{2}'\log(\pphi{2}')})$, so that the process exhibits coexistence while having relative fitness $\frac{\pphi{1}'}{\pphi{2}'}=\gamma$.

The proof of \Cref{theo:5} is based on a stronger version of the coexistence and domination for the dynamical system (see \Cref{extcoex}), which we can then translate to the particle system by means of Theorem \ref{teo1paso}. 
It follows that any improvement on our knowledge of $\DS(h)$ directly improves Theorem~\ref{theo:5}, and that the same ideas could be applied in principle to the system with $m>2$, as soon as the dynamical system is well understood.

\section{Proof of the approximation result}\label{convergenceproofs}

The goal of this section is to prove \Cref{teo1paso}.
As a first step we derive the explicit formula \eqref{eq:galph} for the expected density $g_{\alpha}$ after the epidemic stage on a percolated 3-tree.
Recall that $\TT$ denotes an infinite 3-tree, $\Pe_p$ denotes the site percolation measure on $\TT$ with density $p$, and $\Cr$ denotes the percolation cluster containing a given vertex $r$.

\begin{proof}[{Proof of Proposition \ref{prop1}}]\label{proof:prop1}
	We have
	\begin{equation}\label{eq:eep}
  \textstyle\Ee_p((1-\alpha)^{|\Cr|}\mathds{1}_{|\Cr|>0}) = \sum_{n=1}^\infty (1-\alpha)^n\Pe_p(|\Cr|=n).
  \end{equation}
	Let $A_n$ be the number of possible connected components of size $n$ in a 3-tree rooted at $r$, so that $\Pe_p(|\Cr| = n)=A_n\ts p^n(1-p)^{n+2}$ (notice that $n+2$ is the number of vacant sites surrounding a cluster $\Cr$ of size $n$).
	Noting that a 3-tree is a root connected to three binary trees and recalling that the analog of $A_n$ for a binary tree is given by the Catalan numbers $C_n$, we get $A_0=1$ and $A_{n+1}=\sum_{i=0}^n\sum_{j=0}^{n-i} C_iC_jC_{n-i-j}$.
	Defining the generating functions $A(x)=\sum_{n=0}^\infty A_nx^n$ and $C(x)=\sum_{n=0}^\infty C_nx^n$, the above equation gives
	$A(x)=x\tts C(x)^3+1=\frac18x^{-2}(1-\sqrt{1-4x})^3+1$,
	where we have used the explicit formula for $C(x)$ (see \cite{OEIS02}).
	From this we conclude that the left hand side of \eqref{eq:eep} equals
	\[\textstyle\sum_{n=1}^\infty (1-\alpha)^np^n(1-p)^{n+2}A_n=(1-p)^2\big(A((1-\alpha)p(1-p)) -1 \big) =\tfrac{\left( 1-\sqrt{1-4(1-\alpha)p(1-p)} \right)^3}{8 (1-\alpha)^2p^2}.\qedhere\]	
\end{proof}

The remainder of this section is based on  a quantitative version of the arguments in \cite{DYR09}. 

Assume that \eqref{eq:alphaconv} holds.
The function $h$ has been defined in terms of the behavior of the system when $\GN$ is replaced with an infinite $3$-regular tree, so in our approximation it will be convenient to focus on the vertices whose neighborhoods look locally like a tree. With this in mind define
\[H_N = \{ x\in \GN: G_N\cap B(x,L_N) \text{ is a finite 3-regular tree} \}\]
with $L_N=\log_2(N)/5$. 
From the proof of Lemma 3.2 in \cite{DYR09} we get that
\begin{equation}\label{eq:fromL32DR}
\tfrac1{N}\E(\GN\!\setminus\! H_N)\leq C\tts N^{-3/5}
\end{equation}
for some $C>0$; in particular, the expected density of sites in $H_N$ goes to $1$.
We will use this to control the process locally in balls of radius $L_N$; in fact, as the next result shows, infections coming from further away have a vanishing effect on the system.
Let $\beeta{1}$ be defined similar to $\eeta{1}$, with the difference that for $x\notin H_N$ we set $\beeta{1}(x)=0$ and for $x\in H_N$ the epidemic stage ignores infections coming to $x$ from vertices outside $B(x,L_N)$. We also let $\brho{1}$ the vector of densities $(\brho{1}{1},\ldots,\brho{1}{m})$.
\begin{lemma}\label{midstep}
	For any $\varepsilon>0$ there exists $N_0$ such that for all $N\geq N_0$ and $\rrho{0}$ one has
	\begin{align}
    &\label{ineq:4}|\E(\brho{1}{j}|\GN)-\hh{}{j}(\rrho{0})|\leq \varepsilon+\sfrac{1}{N}|\GN\tsm\setminus\tsm H_N|
    \end{align}
	for each species $j$. Moreover, for all $\rrho{0}$,
	\begin{align}	
	&\label{ineq:2}\E\left(\big|\sfrac{1}{N}|\eeta{1}{j}\cap H_N|-\brho{1}{j}\big|\right)\leq e^{-\alp{j}{N}L_N}.\,
	\end{align}
  where $\eeta{k}{j}$ denotes the set of vertices $x$ such that $\eeta{k}(x)=j$.
\end{lemma}

\begin{proof}
	Pick a vertex $r\in \GN$ uniformly at random and use the definition of $\beeta{1}$ to express $\E(\brho{1}{j}|\GN)$ as
    \[
    \E(\brho{1}{j}|\GN)=\P(\beeta{1}(r)=j|\GN)=\E\Big(\uno{r\in\eeta{1/2}{j}\cap H_N}(1-\alp{j}{N})^{|\Cr^j\cap B(r,L_N)|}|\GN\Big),
    \]
	where $\Cr^j$ is the connected component of type $j$ containing $r$ at time $1/2$.  The event $r\in H_N$ implies that $B(r,L_N)$ is a 3-regular tree, and by the mean-field assumption for the growth stage, at time $1/2$ each vertex is occupied by a type $j$ individual independently with probability $q=\ff{j}(\rrho{0})$. As a consequence, $|\Cr^j\cap B(r,L_N)|$ is the size of the cluster containing $r$ in the percolated 3-regular tree, which we represent as the total amount of individuals of a Galton-Watson process $Z_0,Z_1,\ldots, Z_{L_N}$. More precisely since a 3-regular tree can be seen as a vertex connected to the root of three binary trees, we set the offspring distribution of the first generation of the Galton-Watson process to be a Binomial$[3,q]$ and of all subsequent generations to be a Binomial$[2,q]$,  with  $Z_0=\uno{r\in\eeta{1/2}{j}}$, giving the expression
\begin{equation}
  \begin{aligned}
  \E(\brho{1}{j}|\GN)&=\;\ts\E\big({\bf1}_{\{r\in H_N\}}Z_0(1-\alp{j}{N})^{Z_0+Z_1+\cdots+Z_{L_N}}|\GN\big)\\[4pt]
  &=\;\ts\E({\bf1}_{\{r\in H_N\}}|\GN)\ts\E\big(Z_0(1-\alp{j}{N})^{Z_0+Z_1+\cdots+Z_{L_N}}\big),
  \end{aligned}\label{eq:brhobd}
  \end{equation}
	where the second equality comes from the fact that given the event $r\in H_N$, the variables $Z_0,Z_1,\ldots, Z_{L_N}$ do not depend on the particular realization of $G_N$. For the expression on the right we have the following result concerning Galton-Watson processes, whose proof is postponed to the appendix:
	
	\begin{lemma}
		\label{convarbol}
		Take $\alpha\in(0,1)$ and a Galton-Watson process $Z_0,Z_1,\ldots$ as above. Then, there is a $C'>0$ independent of $\alpha$ such that for all $N$ and all $q\in[0,1]$, 
		\begin{equation}
		\label{princaprox}
		\big|\ts\E\big(Z_0(1-\alp{}{})^{Z_0+Z_1+\cdots+Z_{L_N}}\big)-	g_{\alp{}{}}(q)\big|\leq C' e^{-\alp{}{}L_N}.
		\end{equation}
	\end{lemma}

Using \eqref{princaprox} and the fact that $g_{\alp{j}{N}}$ converges uniformly to $g_{\alp{j}}$, and since the constant $C'$ in \Cref{convarbol} does not depend on $q$, we deduce that for large enough $N$
\[|\E(\brho{1}{j}|\GN)-\hh{}{j}(\rrho{0})|\leq \varepsilon+\P(r\notin H_N|\GN)\]
whence \eqref{ineq:4} follows.

Now we prove \eqref{ineq:2}. Notice that $\brho{1}{j}-\frac{1}{N}|\eeta{1}{j}\cap H_N|$ corresponds by definition to the fraction of vertices $x$ that belong to $H_N$ and which at time $\frac12$ are occupied by an individual of type $j$ that survives the restricted epidemic but not the unrestricted one. In particular, for any such vertex there must be an open path to the boundary of $B(x,L_N)$ used by the unrestricted infection to kill $x$, so we deduce
	\[\E\Big(\Big|\sfrac{1}{N}|\eeta{1}{j}\cap H_N|-\brho{1}{j}\Big|\Big)\;\leq\;(1-\alp{j}{N})^{L_N}\;\leq\;e^{-\alp{j}{N}L_N}.\]
\end{proof}
\smallskip

\begin{proof}[Proof of \Cref{teo1paso}]
Observe first that, since $\delta>0$ is arbitrary and from the uniform continuity of $h$, we only need to prove the statement of the theorem for $k=1$. Even further, it is enough to show that for any fixed $j\in\{1,\ldots,m\}$ and $\delta>0$ we can find $C>0$ as in the statement such that
\begin{equation}
\label{eqj}
\P\!\left(\big|\rrho{1}{j}-\hh{}{j}(\rrho{0})\big|>\delta\>\right)\;\leq\;Ce^{-\alp{j}{N}L_N}.\end{equation}
Define $H_N$ and $\beeta{1}$ as before.
The left hand side of \eqref{eqj} is bounded by
\begin{multline}\label{eq:princ2}
\P\!\left(\big|\rrho{1}{j}-\sfrac{1}{N}|\eeta{1}{j}\cap H_N|\big|>\sfrac{\delta}3\right)+\P\!\left(\big|\sfrac{1}{N}|\eeta{1}{j}\cap H_N|-\brho{1}{j}\big|>\sfrac{\delta}3\right)\\
+\P\!\left(|\brho{1}{j}-\hh{}{j}(\rrho{0})|>\sfrac{\delta}3\right),
\end{multline}
and hence the result will follow after showing that each term on the right hand side is bounded by $Ce^{-\alp{j}{N}L_N}$ for some $C$, independently of $\rrho{0}$.
For the first term on the right hand side of \eqref{eq:princ2} we use Markov's inequality to get
\begin{align*}
\P\!\left(\big|\rrho{1}{j}-\sfrac{1}{N}|\eeta{1}{j}\cap H_N|\big|>\sfrac{\delta}3\right)&\leq\sfrac{3}{\delta}\ts\E\!\left(\big|\rrho{1}{j}-\sfrac{1}{N}|\eeta{1}{j}\cap H_N|\big|\right)\leq\tfrac3{N\delta}\E(\GN\!\setminus\! H_N)\\
&\leq C\tts N^{-3/5}\leq C\tts e^{-\alp{j}L_N},
\end{align*}
for some $C>0$, where we have used \eqref{eq:fromL32DR}.
The second term is similarly bounded by $\sfrac{3}{\delta}\ts\E\!\left(\big|\sfrac{1}{N}|\eeta{1}{j}\cap H_N|-\brho{1}{j}\big|\right)$, for which the estimate follows from \eqref{ineq:2} in \Cref{midstep}.

We turn now to the third term on the right hand side of \eqref{eq:princ2}.
It will be convenient to condition on the realization of $\GN$: introducing the notations $\P_{\GN}=\P(\cdot|\GN)$ and $\E_{\GN}=\E(\cdot|\GN)$ we may estimate this term as
\begin{multline}\label{mid2step}
\P\!\left(|\brho{1}{j}-\hh{}{j}(\rrho{0})|>\sfrac{\delta}3\right)\\
\qquad\leq\E\!\left(\P_{\GN}\!\left(|\brho{1}{j}-\E_{\GN}(\brho{1}{j})|>\sfrac{\delta}6\right)\right)+\E\!\left(\P_{\GN}\!\left(|\E_{\GN}(\brho{1}{j})-\hh{}{j}(\rrho{0})|>\sfrac{\delta}6\right)\right).
\end{multline}
From \eqref{ineq:4} in \Cref{midstep} with $\varepsilon=\delta/12$ we can estimate the second term for large $N$ as
\[\P\tsm\left(|\E_{\GN}(\brho{1}{j})-\hh{}{j}(\rrho{0})|>\sfrac{\delta}6\right)\leq\P\tts(\sfrac{1}{N}|\GN\tsm\setminus\tsm H_N|>\sfrac{\delta}{12})\leq\sfrac{12}{\delta}\ts\E(\sfrac{1}{N}|\GN\tsm\setminus\tsm H_N|),\]
which is bounded by $C\tts e^{-\alpha_N(j)L_N}$ as above.
So what remains is to bound the first term on the right hand side of \eqref{mid2step}.
We focus on the inner conditional probability, which is bounded by 
\[\sfrac{36}{\delta^2}\ts\E_{\GN}\!\left((\brho{1}{j}-\E_{\GN}(\brho{1}{j}))^2\right).\]
Setting $r_j(x)=\E_{\GN}(\beeta{1}{j}(x))$, we write
\begin{equation*}
\E_{\GN}\!\left((\brho{1}{j}-\E(\brho{1}{j}))^2\right)
=N^{-2}\tts\E_{\GN}\!\left({\textstyle\sum_{x,y\in\GN}}(\uno{x\in\eeta{1}{j}}-r_j(x))(\uno{y\in\eeta{1}{j}}-r_j(y))\right),
\end{equation*}
Since the events $\{x\in\beeta{1}{j}\}$ and $\{y\in\beeta{1}{j}\}$ are independent for $x,y\in H_N$ with $d(x,y)>2L_N$, we may bound the right hand side by
\[N^{-2}\tts\E_{\GN}\big(\left|\left\{(x,y)\in H_N\times H_N,\;d(x,y)\leq 2L_N\right\}\right|\big)+N^{-2}\E_{\GN}(|\GN\times\GN\setminus H_N\times H_N|);\]
the first term is bounded by $N^{-2}\E_{\GN}(\sum_{x\in G_N}\left|B(x,2L_N)\right|)=N^{-2}\tts(3N\cdot N^{2/5})=3N^{-3/5}$, while the second one is bounded by $2N^{-2}\E_{\GN}(|\GN\setminus H_N|)$.
Taking expectation, we see that the first term on the right hand side of \eqref{mid2step} is bounded by
\[2N^{-3/5}+2N^{-2}\E(|\GN\setminus H_N|)\leq C\tts e^{-\alpha_N(j)L_N}\]
as needed, finishing the proof.
\end{proof}

\section{Proof of the main result}\label{multitypeproofs}

\subsection{Interior-recurrent sets}

As discussed at the end of Section~\ref{conMMM}, our approach to prove \Cref{theo:5} consists in using \Cref{teo1paso} to show that the particle system tracks the behavior observed for the dynamical system in Theorem \ref{teodinamico}. 
However, if \Cref{teo1paso} is applied directly to try to handle the stochastic system for a number of steps which depends on $N$, one loses control on the constant $C$ appearing in the estimate (and in fact we expect it to grow fast with $N$ due to the chaotic behavior of the dynamical system).
In order to fix this problem we introduce the notion of interior-recurrent sets, which are in essence subsets of the state space that are visited by the dynamical system repeatedly in a bounded number of steps, and which we will use to divide the trajectories of the stochastic system into excursions between hitting times, so that \Cref{teo1paso} can be used on each individual excursion.

\begin{definition}\label{inttrap}
	We say that a set $A\subseteq [0,1]^2$ is {\sl interior-recurrent} for $\DS(h)$ if there are $0<\delta'<\delta$ and $\bar{k}\in\N$ such that
	\begin{enumerate}[label=(\roman*)]
		\setlength\itemsep{0.3em}
		\item $\forall\tts p\in A$, $d(p,A^c)>\delta\;\Longrightarrow\;d(\hh(p),A^c)\geq\delta'$,\label{inttrap1}
		\item $\forall\tts p\in A$, $d(p,A^c)\leq \delta\;\Longrightarrow\;d(\hh{k}(p),A^c)\geq\delta'$ for some $k\leq\bar{k}$.\label{inttrap2}
	\end{enumerate}
\end{definition}

\noindent In words, a set $A$ is interior-recurrent if the dynamical system cannot exit its interior using jumps larger than a certain size $\delta$ and if every time it gets to a distance smaller than $\delta$ to the boundary, it takes a bounded number of steps for it to go back to a certain subset of $A$ which is bounded away from its boundary. 
The next proposition shows that, thanks to the approximation result Theorem \ref{teo1paso}, the control on $\DS(h)$ furnished by interior-recurrent sets can be transferred to the particle system.

\begin{prop}\label{trapmmm}
	Let $(\eeta{k})_{k\in\N}$ be the {MMM} with parameters satisfying the conditions in \Cref{teo1paso}, and assume that its initial condition $\rrho{0}$ lies within an interior-recurrent set $A$ with parameters $\delta$, $\delta'$ and $\bar{k}$. Then there is a $C>0$ depending only on $\delta'$ and $\bar{k}$, such that
	\begin{equation}
	\label{eq4theo5}
	\P\!\left(\rrho{k}\notin A,\;\forall k\in\{1,2,\ldots,\bar{k}\}\right)\;\leq\;Ce^{-\ualpha_N\log_2(N)/5}.
	\end{equation}
\end{prop}

\begin{proof}
	From \Cref{teo1paso} there is a $C>0$ as in the statement such that for any $k\leq\bar{k}$ and $\rrho{0}$, we have
	\begin{equation*}
	\P\!\left(\big\|\rrho{k}-\hh{k}{}(\rrho{0})\big\|>\delta'\right)\;\leq\;Ce^{-\ualpha_N\log_2(N)/5}.
	\end{equation*}
  Now we use the interior-recurrence of $A$. If $d(\rrho{0},A^c)>\delta$ then $d(\hh(\rrho{0}),A^c)>\delta'$, so the left hand side of \eqref{eq4theo5} is bounded by
	\[\P\big(\rrho{1}\notin A\big)\;\leq\;\P\big(\|\rrho{1}-\hh(\rrho{0})\|>\delta'\big)\;\leq\;Ce^{-\ualpha_N\log_2(N)/5}.\]
  Otherwise, if $d(\rrho{0},A^c)\leq\delta$, then there is a $k\leq\bar{k}$ such that $d(\hh{k}{}(\rrho{0}),A^c)>\delta'$, and the same argument shows that the left hand side of \eqref{eq4theo5} is bounded by the required amount.
\end{proof}

With the concept of interior-recurrent sets in hand, we can now state the more precise version of \Cref{teodinamico}, which gives stronger versions of coexistence and domination for $\DS(h)$ and which, together with \Cref{trapmmm}, will yield \Cref{theo:5}. In order to state it we define $\pb\in[0,1]^2$ as the vector of maximum possible densities achieved after the epidemic stage, that is
\begin{equation}\label{barp}\pb{i}=\sup\nolimits_{x\in[0,1]}\gg{i}(x).\end{equation}

\begin{theo}\label{extcoex}
	Let $\mathcal{F}_1$ and $\mathcal{F}_2$ be as in \Cref{teodinamico} and consider the dynamical system $\DS(h)$ with $1<\pphi{1}<\pphi{2}$.
	\begin{enumerate}[label=\textup{(\roman*)}]
		\item \textup{(Survival)} There are $\bar{c}$ and $\varepsilon>0$ such that for any $0<c<\bar{c}$ the set $[0,1]\times[c,\pb{2}+\varepsilon]$ is interior-recurrent.
		\item \textup{(Coexistence)} Assume that $\pphi{1}>\mathcal{F}_1(\alp{2},\pphi{2})$ and $\pphi{2}>2\log 2$.
		Then there are $\bar{c}_1,\bar{c}_2>0$ and $\varepsilon_1,\varepsilon_2>0$ such that for any $c_1\leq\bar{c}_1$ and $c_2\leq\bar{c}_2$, the set $A=[c_1,\pb{1}+\varepsilon_1]\times[c_2,\pb{2}+\varepsilon_2]$ is interior-recurrent.
		\item \textup{(Domination)} Assume that $\pphi{1}<\mathcal{F}_2(\alp{2},\pphi{2})$. Then, there are $\gamma_1,\gamma_2\in(0,1)$ and an interior-recurrent set $B$ with parameter $\bar{k}=1$ such that for all $p\in B$
		\begin{equation}\label{condB}(1-\alp{1})\ff{1}(\pp)\leq \gamma_1\pp{1}\quad\text{and}\quad\gamma_2<\pp{2}.\end{equation}
		Furthermore, for any $l_2>0$ there is a $k'\in\N$ such that for any $\pp{}{0}$ satisfying $l_2<\pp{2}{0}$, there is $k\leq k'$ for which $\hh{k}{}(\pp{}{0})$ is an interior point of $B$.
	\end{enumerate} 
\end{theo}

Before turning to the proofs of Theorems \ref{theo:5} and \ref{extcoex}, we show how the existence of the recurring sets described above implies the coexistence and domination behaviors of the system as given in \Cref{teodinamico}:

\begin{proof}[Proof of \Cref{teodinamico}]
	Since after one iteration the dynamical system is upper bounded by $\pb$ we will assume that $\pp{}{0}$ also satisfies this bound. Under the coexistence assumptions, \Cref{extcoex} states that there is a compact interior-recurrent set $A\subseteq(0,1)^2$ containing $\pp{}{0}$, and by definition this implies that the orbit of $\DS(h)$ is contained in $A_{\bar{k}}\colonequals\cup_{l=0}^{\bar{k}}\hh{l}{}(A)$, which is also compact. Since $A_{\bar{k}}\subseteq(0,1)^2$ (otherwise it would contain an orbit that never returns to $A$) we deduce that $\liminf_{k\to\infty}\hh{k}{i}(p)>0$ for $i=1,2$. The same argument with $[0,1]\times[c,1]$ instead of $A$ gives survival of type $2$.

  Under the domination assumptions, \Cref{extcoex} states that the orbit of $\DS(h)$ eventually reaches an interior-recurrent set $B$ with parameter $\bar{k}=1$, which satisfies \eqref{condB} for some $\gamma_1,\gamma_2\in(0,1)$, and since $\bar{k}=1$ the system never leaves the set $B$. Since $\gamma_2<\pp{2}$ for $p\in B$, we deduce that $\liminf_{k\to\infty}\hh{k}{2}(p)\geq\gamma_2>0$; similarly, since $\hh{}{1}(\pp)\leq(1-\alp{1})\ff{1}(\pp{1})\leq\gamma_1\pp{1}$ for $p\in B$ (the first inequality follows from comparing with a system where we let the epidemic attack but not spread), we deduce that $\lim_{k\to\infty}\hh{k}{1}(p)=0$. This shows that type 2 dominates.
\end{proof}

\begin{proof}[Proof of \Cref{theo:5}]
	Consider the two-type {MMM} under the conditions of \Cref{theo:5intro} and observe that the hypothesis $l_1<\rrho{0}{1}<u_1$ and $l_2<\rrho{0}{1}<u_2$ imply that
	\[l_1'<\hh{}{1}(\rrho{0})\leq\pb{1}\quad\text{and}\quad l_2'<\hh{}{2}(\rrho{0})\leq\pb{2}\]
	for some $l_1'$ and $l_2'$ depending on $l_1,l_2,u_1$, and $u_2$. In particular, using \Cref{teo1paso} we can safely assume that $l_1'<\rrho{0}\leq\pb{1}+\varepsilon_1$ and $l_1'<\rrho{0}\leq\pb{2}+\varepsilon_2$ for sufficiently small $\varepsilon_1$ and $\varepsilon_2$. Assume first that the parameters of the model satisfy the coexistence conditions of \Cref{extcoex}. As in the previous proof, these conditions ensure that the set $[c_1,\pb{1}]\times[c_2,\pb{2}]$ is interior-recurrent for sufficiently small $c_1$ and $c_2$. In particular we can choose $c_1<l_1$ and $c_2<l_2$ and hence it contains any initial condition $\rrho{0}$. Let $\sigma_n$ denote the $n$-th return time of the dynamical system to $A$. By \Cref{trapmmm} we have $\P(\sigma_1>\bar{k})\leq Ce^{-\ualpha_N\log_2(N)/5}$ for some $C>0$ which is independent of the initial condition. By the strong Markov property we get
	\[\P\big(\sigma_1\leq\bar{k},\;\sigma_2-\sigma_1\leq\bar{k},\ldots,\;\sigma_{n}-\sigma_{n-1}\leq\bar{k}\big)\;\geq\;(1-Ce^{-\ualpha_N\log_2(N)/5})^n.\]
  The event on the left hand side implies in particular that $\sigma_n<\infty$ a.s., but since $\sigma_n\geq n$ it follows that $\rrho{k}\in [c_1,\pb{1}]\times[c_2,\pb{2}]$ for some $k\geq n$. Since both species have to be alive to lie within this set, on this event both $\tt{1}$ and $\tt{2}$ must be larger than $n$, so \eqref{esperanzacoex} follows by choosing $n=e^{c_1\ts\ualpha_N\tsm\log(N)}$ for some $c_1<\frac{1}{5}$. For the general case $\pphi{1}<\pphi{2}$ we can use the same argument with $[0,1]\times[c,1]$ replacing $[c_1,\pb{1}]\times[c_2,\pb{2}]$, giving \eqref{esperanzadom2}.
	
	Suppose now that the parameters satisfy the domination conditions of \Cref{extcoex} so that there is an interior-recurrent set $B$ with parameter $\bar{k}=1$ which satisfies \eqref{condB}. Assume first that $\rrho{0}$ lies in the interior of $B$. We will explain later how to treat the case in which the initial condition is not in the interior of $B$.
	
	Since $\bar{k}=1$, \Cref{inttrap} implies that regardless of the value of $d(\rrho{0},B^c)$ we have $d(\hh(\rrho{0}),B^c)>\delta'$, so \Cref{teo1paso} gives some $C>0$ depending only on $B$ such that $\P(\rrho{1}\notin B)\,\leq\,\P\!\left(\delta'<\big\|\rrho{1}-\hh(\rrho{0})\big\|\right)\,\leq\,Ce^{-\ualpha_N\log_2(N)/5}$. Since the bound is uniform over $\rrho{0}\in B$, an application of the strong Markov property gives that, for any $n\in\N$,
	\begin{equation}\label{ineq1}\P(\rrho{k}\in B~\forall k\leq n)\;\geq\;(1-Ce^{-\ualpha_N\log_2(N)/5})^n.
	\end{equation}
	Noticing that $\gamma_2<\pp{2}$ for all $p\in B$ we deduce that the event on the left hand side implies $\tt{2}\geq n$, so \eqref{esperanzadom2} follows by choosing $n=e^{c_1\ts\ualpha_N\tsm\log(N)}$ for some $c_1<\frac{1}{5}$ as in the coexistence scenario.
	
	To deduce \eqref{esperanzadom1} observe that under the assumption $\rrho{0}\in B$ the number of type 1 individuals at time 1 is dominated by a Poisson random variable with parameter $(1-\alp{1}{N})\ff{1}(\rrho{0})$, which is less than $\gamma_1\rrho{0}$. From this one sees that on the event $\mathcal{E}_n=\{\rrho{k}\in B~\forall k\leq n\}$, the process $(N\rrho{k})_{k\leq n}$ is stochastically dominated by a subcritical Galton-Watson process starting with $\rrho{0}N$ individuals and with offspring distribution Poisson$[\gamma_1]$. By \eqref{ineq1} and standard branching processes results we get
	\begin{equation}\label{ineq2}
	\P\left(\tt{1}\geq n\right)\leq\P\left(\mathcal{E}_n\cap\{\tt{1}> n\}\right)+\P(\mathcal{E}_n^c)
	\leq \rrho{0}N\gamma_1^{n}+1-(1-Ce^{-\ualpha_N\log_2(N)/5})^n
	\end{equation}
  and then \eqref{esperanzadom1} follows by taking $n=c_1'\log N$ for some small $c_1'>0$. 
	
	Suppose now that $\rrho{0}$ is not an interior point of $B$ and observe that under the conditions of \Cref{theo:5intro} there is some $l_2$ such that $l_2<\rrho{0}{2}$ and hence from \Cref{extcoex} there is some $k'\in\N$ depending only on $l_2$ such that $\hh{k}(\rrho{0})$ is an interior point of $B$ for some $k\leq k'$. Observing that the set $[0,1]\times[l_2,\pb{2}]$ is compact and the function $\hh{k}$ is continuous for all $k\leq k'$ there is $\varepsilon$ such that $d(\hh{k}(p_0),B^c)>\varepsilon$ for any $k\leq k'$ and $p_0\in [0,1]\times[l_2,\pb{2}]$. Using \Cref{teo1paso} there is a $\bar{C}$ depending on $k'$ and $\varepsilon$ such that, for the particular value of $k$ such that $\hh{k}(\rrho{0})\in B$,
	\begin{equation}\label{ineq3}\P(\rrho{k}\notin B)\;\leq\;\P\big(\big\|\rrho{k}-\hh{k}(\rrho{0})\big\|>\varepsilon/2\big)\;\leq\;\bar{C}e^{-\ualpha_N\log_2(N)/5},\end{equation}
	so the general proof of \eqref{esperanzadom1} and \eqref{esperanzadom2} follows from restricting to the event on the left hand side above and restarting the process at time $k$.
\end{proof}

The rest of this section is devoted to the proof of \Cref{extcoex}, which is rather long and technical, so we divide it into three parts. In \Cref{preliminaries} we present some preliminary notation and functions which will be used to facilitate the analysis of the trajectories of $\DS(h)$, as well as some technical results about them. Using these results we prove the coexistence part of the theorem in \Cref{coex}, and the domination part in \Cref{ext}.

\subsection{Preliminaries}
\label{preliminaries}

We begin this section by decomposing the function $\gg{}$, $\alpha\in[0,1)$, as
\begin{equation}
\gg{}(x)=(1-\alp)x \GG{}(x)^3\qquad\text{with}\quad \GG{}(x)=\frac{1-\sqrt{1-4(1-\alpha)x(1-x)}}{2(1-\alpha)x}.\label{eq:Galpha}
\end{equation}

\begin{lemma}
	\label{tecnico1}
	The function $\GG{}:[0,1]\to[0,1]$ satisfies the following:
	\begin{enumerate}[label=\textup{(\arabic*)}]
		\item For $\alpha=0$ it is given as $G_0(x)=1$ for $x\leq 1/2$ and $\sfrac{1-x}{x}$ if $x>1/2$.
		\item It is decreasing as a function of both $\alpha$ and $x$, with $\GG{}(0)=1$ and $\GG_\alpha(1)=0$ for all $\alpha \in [0,1)$.
		\item As $\alp\to1$, it converges monotonically to $G_1(x)\coloneqq1-x$.		
	\end{enumerate}
\end{lemma} 

We omit the simple proof of this result.
Recall that $\pb\in[0,1]^2$ was defined as
\begin{equation*}\pb{i}=\sup\nolimits_{x\in[0,1]}\gg{i}(x).\end{equation*}

By (2) of the last lemma $\gg{}\leq (1-\alp)g_0$, from which it follows that $\pb{i}\leq\frac{1-\alp{i}}{2}$.
By definition, except maybe for the initial value $\pp{}{0}$, the orbit of $\DS(h)$ lies within $[0,\pb{}]$; the next result provides control on the behavior of $\gg{}$ on that interval:

\begin{prop}
	\label{maxg}
	$\gg{}$ attains its global maximum at a single value $x_0\in[0,1/2]$. This value is characterized as the solution of $\GG{}(x_0)=x_0+\sfrac{1}{2}$ and satisfies:
	\begin{enumerate}[label=\textup{(\arabic*)}]
		\item If $\alpha>0$, $x_0$ is the only critical point of $g_{\alpha}$ in $[0,1]$.
		\item If $\pphi{i}<\betac$, then for any $\pp$ with $\pp{i}\leq \pb{i}$ we have $\ff{i}(\pp)<x_0$. In particular, $\gg{i}'\circ\ff{i}(\pp)>0$ for all $\pp\in[0,\pb{1}]\times[0,\pb{2}]$. 
	\end{enumerate}
\end{prop}

Even though $\gg{}$ is not monotone, the last result still yields enough information about the growth of $\hh$:

\begin{prop}
	\label{crecimiento}
	For each $i=1,2$ define $\ll{i}:[0,1]^2\rightarrow\R^+$ as $\ll{i}(p)=\hh{}{i}(p)/\pp{i}$. Then:
	\begin{enumerate}[label=\textup{(\arabic*)}]
		\item The function $\ff{1}(p)$ is increasing in $\pp{1}$ and decreasing in $\pp{2}$.
		\item The function $\ll{1}(p)$ is decreasing in $\pp{1}$.
		\item If $\pphi{1}<\betac$, then  $\hh{}{1}(p)$ is increasing in $\pp{1}$ and decreasing in $\pp{2}$. In particular, in this case $\ll{1}$ is also decreasing in $\pp{2}$.
	\end{enumerate}
\end{prop}

We are interested in $\ll_i$ because, since $\hh{}{i}(p)=\ll{i}(p)\pp{i}$, it is enough to bound $\ll{i}$ in order to get exponential growth or decay of the density of a species. This is what we do in the next result.

\begin{prop}
	\label{propc}
	Assume that $\pphi{2}>\pphi{1}>1$. For small $\varepsilon> 0$ define $\ke$ as the unique solution of
	\[\gg{1}(1-e^{-\bet{1} \ke})\quad=\quad (1-\varepsilon)\ke\]
	in $(0,1)$. There are $\bar{c},\varepsilon,\varepsilon'>0$ small such that for all $c\leq \bar{c}$:
	\begin{enumerate}[label=\textup{(\arabic*)}]
		\item \label{propc1}For all $0<\pp{1}<\ke$ it holds that
		\begin{align}
		\label{propc11} \pp{2}\in(0,c)&~\Longrightarrow~\ll{2}(p)>1+\varepsilon',\\
		\label{propc12} \pp{2}\in(c,\pb{2})&~\Longrightarrow~\hh{}{2}(p)>(1+\varepsilon')c.
		\end{align}
		\item \label{propc2} Under the additional assumption $\pphi{2}>\betac$, the property in (1) holds for all $\pp{1}>0$.
		\item \label{propc3} If $\pphi{1}<\betac$, then:
		\begin{align}
		\label{propc31}\pp{1}\in(0,\ke)&~\Longrightarrow~\hh{}{1}(p)\leq(1-\varepsilon')\ke.\\
		\label{propc32} \pp{1}\in(\ke,\pb{1})&~\Longrightarrow~\ll{1}(p)\leq 1-\varepsilon'.
		\end{align}
	\end{enumerate}
\end{prop}

Properties \ref{propc1} and \ref{propc2} state that when the stronger species starts at a low density, it grows exponentially until it reaches a certain threshold value $c$, which becomes a lower bound for its density from that time onwards. Property \ref{propc3}, on the other hand, states that if the fitness of the weaker species is below $\betac$, then its density decays exponentially until it reaches a trapping set $[0,\ke]$.

The proofs of the last three propositions are mostly calculus, so we defer them to the appendix.

\subsection{Proof of Theorem \ref{extcoex}(ii)}
\label{coex}

As we just discussed, \Cref{propc} already provides a good control on the behavior of the stronger species, so our main focus will be on the weaker one. Assuming that the coexistence conditions of \Cref{extcoex} are satisfied, our approach consists in analyzing the dynamical system when the density of the weaker species is at low values. In that situation we will approximate $\hh$ by a simpler function $\bh$, and show that for this particular dynamical system the density $\pp{1}$ tends to grow on average.

The approximating map $\bh:\R\times[0,1]\rightarrow\R\times[0,1]$ which we will use is the linearization of $\hh$ in its first component,
\[\bh(p)\;=\;\left(\begin{array}{c}\hh{}{1}(0,\pp{2})+\pp{1}\frac{\partial \hh{}{1}}{\partial \pp{1}}(0,\pp{2})\\ \hh{}{2}(0,\pp{2})\end{array}\right)\;=\;\left(\begin{array}{c}\pphi{1} \pp{1}\frac{1-e^{-\bet{2} \pp{2}}}{\bet{2} \pp{2}}\\ \hh{}{2}(0,\pp{2})\end{array}\right).\]
The next result states that this approximation is good uniformly in $\pp{2}$:

\begin{prop}\label{lem:1}
	For any fixed $k\in\N$ we have $\lim_{\pp{1}\to 0}\frac{\bh{k}{1}(\pp)}{\hh{k}{1}(\pp)}=1$ uniformly in $\pp{2}\in [0,1]$.
\end{prop}

The following function will be used in the proof of the proposition and in later results:
\begin{align}\label{psi}
	\psi(x)=\frac{1-e^{-x}}{x}.
\end{align}

\begin{proof}[Proof of \Cref{lem:1}]
	Let $\Sigma_k(p)=\bet{1}\hh{k}{1}(p)+\bet{2}\hh{k}{2}(p)$ and $\overline{\Sigma}_k(p)=\bet{2}\bh{k}{2}(p)$. Using these values and the definition of $\hh$ and $\bh$ it is fairly simple to see that
	\begin{equation}
	\label{eqa1}	
	\frac{\bh{k}{1}(p)}{\hh{k}{1}(p)}\;=\;\frac{\bh{k-1}{1}(p)}{\hh{k-1}{1}(p)}\ts\frac{\psi(\overline{\Sigma}_{k-1}(p))}{\psi(\Sigma_{k-1}(p))}\ts(\GG{1})^{-3}\circ\ff{1}\circ\hh{k-1}{1}(p).
	\end{equation}
	The function $\psi$ is uniformly continuous and bounded away from 0 for $x\in[0,1]$.
  Noticing that $\Sigma_k(p)$ and $\overline{\Sigma}_k(p)$ converge to the same value as $\pp{1}\to 0$ and in view of (2) of \Cref{tecnico1}, the last two factors on the right hand side of \eqref{eqa1} converge to $1$ uniformly, so $\bh{k}{1}(p)/\hh{k}{1}(p)$ converges to $1$ uniformly if $\bh{k-1}{1}(p)/\hh{k-1}{1}(p)$ does. Since $\bh{0}{1}(p)=\hh{0}{1}(p)=\pp{1}$, the result follows by repeating the argument $k$ times.
\end{proof}

Thanks to this proposition we can approximate $\hh$ by $\bh$ whenever $\pp{1}$ is small enough, independently of the value of $\pp{2}$. 
The resulting dynamical system $(\qq{}{k})_{k\in\N}$ can be realized by first running the one-dimensional \MM for type 2 by itself, and then using its trajectory to compute the values of $\qq{1}{n}$ as
\begin{equation}{\small
\label{prod}\qq{1}{n}\;=\;\qq{1}{0}\,\prod_{k=0}^{n-1}\pphi{1}\psi\big(\bet{2}\qq{2}{k}\big)\;=\;\qq{1}{0}\left(\pphi{1}\pbhi{n}(\qq{2}{0})\right)^n\quad\text{with}\quad\pbhi{n}(x)\;=\;\left(\prod_{k=0}^{n-1}\psi\big(\bet{2}\hh{k}{2}(0,x)\big)\right)^{1/n},}
\end{equation}
where $\psi(x)$ is given in \eqref{psi}. This suggests that it will be useful to study
\begin{equation}
\underline{\pbhi}(x)\;\coloneqq\;\liminf_{n\to\infty}\pbhi{n}(x).\label{eq:defpbphi}
\end{equation}
In view of \eqref{prod}, $\pphi{1}\underline{\pbhi}(\qq{1}{0})$ can be interpreted as the average growth of type 1 when taking into account the effect of type 2. In order to control this growth we define $\etac$ to be the smallest possible value of $\underline{\pbhi}$, that is
\begin{equation}
\etac=\inf\nolimits_{x\in[0,\pb{2}]}\underline{\pbhi}(x),\label{eq:defeta}
\end{equation}
where the infimum is taken over $[0,\pb{2}]$ since by \eqref{barp} after one iteration of the system the process gets trapped in $[0,\pb{2}]$. The following result shows that a good control on $\etac$ allows us to make $\qq{1}{k}$ as large as we want:

\begin{lemma}\label{lem:2}
	If $\pphi{1}\etac>1$, then for all $M>0$ there exists $\bar{k}\in \N$ satisfying the following property:
	for all $\qq{2}{0}\in [0,\pb_2]$, there is a $0\leq k\leq \bar{k}$ such that
	\begin{equation}
	\label{eq:3}
	\prod_{j=0}^{k-1}\pphi{1}\psi(\bet{2}\qq{2}{j})\;>\;M.
	\end{equation}
\end{lemma}

\begin{proof}
	From the hypothesis we know that there exists $\delta>0$ such that $\pphi{1}=\sfrac{1+2\delta}{\etac}$. Taking $\varepsilon>0$ small enough such that $(1-\varepsilon)(1+2\delta) > 1+\delta$, for each $\qq{2}{0}$ we can find $\underline{k}\in\N$ such that for all $k\geq \underline{k}$ 
	\[\label{eq:2}
	\pphi{1}\pbhi{k}(\qq{2}{0})\;>\;(1-\varepsilon)	\pphi{1}\underline{\pbhi}(\qq{2}{0})\;\geq\;(1-\varepsilon)\pphi{1}\etac\;>\;1+\delta,
	\]
	where the first inequality follows from the definition of $\underline{\pbhi}$. Using the definition of $\pbhi{k}$ we obtain 
	$\pphi{1}\left(\prod_{j=1}^{k-1}\psi(\bet{2} \qq{2}{j})\right)^{1/k}>1+\delta$ for all $k\geq \underline{k}$.
	In particular we find that for each $\qq{2}{0}$ there is some $k\geq\underline{k}$ such that
	$\prod_{j=0}^{k-1}\pphi{1}\psi(\bet{2}\qq{2}{j})\;>\;M$.
	For $k$ fixed call $O_k$ the set of all $\qq{2}{0}$ satisfying the last inequality for that given value of $k$. From the continuity of $\bh$ and $\psi$ each $O_k$ is open, and from the previous argument each $\qq{2}{0}$ belongs to some $O_k$, so $(O_k)_{k\in\N}$ is an open cover of $[0,\pb{2}]$, which necessarily contains a finite subcover. Taking $\bar{k}$ to be the largest index of the subcover yields the result.
\end{proof}

The next result shows that if $\pphi{1}\etac>1$ then after the species 1 density gets above a certain threshold parameter $c$, it cannot stay below $c$ for more than $\bar{k}$ consecutive steps afterwards.
The idea is simple: as long as the trajectory of $\pp{1}{k}$ stays small then the system is well approximated by $\DS(\bh)$, but the last proposition says that the first component of this system gets large, which hints at a contradiction.
This will be helpful below in showing (ii) in the definition of interior-recurrence for a suitable set.

\begin{prop}\label{theo:1}
	Suppose that $\pphi{1}\etac>1$. There is a $\bar{c}>0$ satisfying the following: for all $c\leq\bar{c}$ we can find $\bar{k}\in\N$ such that for all $n\in\N$
	\begin{equation}
	\label{jumpc}
	c\leq \pp{1}{n}<\pb{1}\quad\Longrightarrow\quad\exists k\leq\bar{k}\text{ such that }\pp{1}{n+k}>\tfrac{3}{2r}c
	\end{equation}
  with $r=\inf_{p\leq\pb{}}\ll{1}(p)$.
\end{prop}

\begin{proof}
	Let $M=2/r^2$, choose $\bar{k}$ as in \Cref{lem:2} for that value of $M$ and use the uniform convergence proved in \Cref{lem:1} to choose $\delta_0>0$ such that
	\begin{equation}\label{eq:1}
	\pp{1}<\delta_0 \quad\Longrightarrow\quad {\bh{k}{1}(\pp)}/{\hh{k}{1}(\pp)} <{4}/{3}\;\;\;\forall \pp{2} \in [0,1],\;\;\forall k=1,\dotsc,\bar{k}.
	\end{equation}
	Define now $\bar{c}=\frac{2\delta_0}{3}$.
  We prove \eqref{jumpc} by contradiction as follows. Choose $c<\bar{c}$ and suppose that for some $n\in\N$ we have $\pp{1}{n}\geq c>\pp{1}{n+1}$ and that  there is no $k\leq\bar{k}$ such that $\pp{1}{n+k}>3c/(2r)$.
  From our choice of $\bar{c}$ we know each $\pp{1}{n+k}$ is smaller than $\delta_0$, so from \eqref{eq:1}, for each $k\leq\bar{k}$ we have
	\begin{equation}\label{eq:4}
	\textstyle\pp{1}{n+k}=\hh{k}{1}(\pp{}{n})\geq \tfrac{3}{4}\bh{k}{1}(\pp{}{n})=\tfrac{3}{4}\tts\pp{1}{n+1}\!\prod_{j=0}^{k-1}\frac{\pphi{1}(1-e^{-\bet{2} \qq{2}{j}})}{\bet{2}\qq{2}{j}}.
	\end{equation}
	However, for the specific value of $k$ given in \Cref{lem:2} with initial condition $\pp{1}{n+1}$, we can bound the right hand side in \eqref{eq:4} from below by $3\pp{1}{n+1}/(2r^2)$. This is a contradiction with our assumption $\pp{1}{n+k}<3c/(2r)$ because
	\begin{equation}\label{eq:5}
	\pp{1}{n+k}>\tfrac{3}{2r^2}\pp{1}{n+1}=\tfrac{3}{2r^2}\ll{1}(\pp{1}{n})\pp{1}{n}\geq\tfrac{3}{2r^2}rc=\tfrac{3}{2r}c,
	\end{equation}
	where the last inequality follows from the definition of $r$ and the assumption $\pp{1}{n}\geq c$.
\end{proof}

Using the tools developed so far we can now prove the coexistence statement of \Cref{extcoex}:

\begin{proof}[Proof of \Cref{extcoex}(ii)]
	Fix $r=\inf_{p\leq\pb{}}\ll{1}(p)$ and take $\bar{c}_1=\bar{c}$ as in \Cref{theo:1} so that \eqref{jumpc} holds for all $c_1\leq \bar{c}_1$. Next, observe that from the assumption $\pphi{2}>\betac$ we can take $\bar{c}_2$ small so that the statement of  \Cref{propc} holds for all $c_2\leq \bar{c}_2$. To see that the set is interior-recurrent, notice that from \eqref{propc12} in \Cref{propc}, for any $\pp{2}\in(c_2,\pb{2})$ we have $\hh{}{2}(\pp{2})>(1+\varepsilon')c_2$ independently of $\pp{1}$, so $\pp{2}{n}$ never goes below $c_2$. In particular both requirements for interior-recurrent are satisfied with $\bar{k}=1$ in the second component. To deduce the same for the first component notice that from the definition of $r$, we have that $\pp{1}>\frac{c_1}{r}$ implies that $\pp{1}{1}>c_1$, and from \Cref{theo:1} there is $\bar{k}$ such that
	$\frac{c_1}{r}>\pp{1}\geq c_1$ implies that there is a $k\leq\bar{k}$ such that $\pp{1}{k}>\tfrac{3}{2r}c_1$, so both requirements for interior-recurrence are satisfied in this component as well. Observe that we have shown that $[c_1,\pb{1}]\times[c_2,\pb{2}]$ is interior-recurrent, but since all the functions involved are continuous and the set is compact we can extend this property to $[c_1,\pb{1}+\varepsilon_1]\times[c_2,\pb{2}+\varepsilon_2]$ (maintaining the same $\bar{k}$) provided $\varepsilon_1$ and $\varepsilon_2$ are small enough.
\end{proof}
 
In order to finish the proof of \Cref{extcoex}(ii) we need to introduce the function $\mathcal{F}_1$ explicitly and explain how the condition $\pphi{1}>\mathcal{F}_1(\alp{2},\pphi{2})$ is sufficient to conclude that $\pphi{1} \etac>1$. 
To do so, define $P_c,P_f\in(0,1)$ as the only critical point and the only positive fixed point of $h(0,\cdot)$, respectively. The fact that $h(0,\cdot)$ has a unique critical point (which is a maximum of the function) follows from \Cref{maxg} and the fact that $x\rightarrow 1-e^{-\bet{2}x}$ is increasing, while the existence of a unique positive fixed point can be proved analogously to the existence and uniqueness of $\ke$ in \Cref{propc} since $\pphi{2}>1$. Using once again \Cref{maxg}, $P_c$ and $P_f$ satisfy
\begin{equation}\label{PcPf}\GG{2}(1-e^{-\bet{2} P_c})=\tfrac{3}{2}-e^{-\bet{2} P_c},\quad\text{ and }\quad \gg{2}(1-e^{-\bet{2} P_f})=P_f.\end{equation}
The two points are related to $\etac$ in the following way:
\begin{itemize}[itemsep=6pt,leftmargin=15pt]
	\item Suppose first that $P_f\leq P_c$, which is equivalent to $h(0,P_c)\leq P_c$. In this case, starting from any initial condition $\pp{2}{0}\in(0,1)$ we have $\pp{2}{1}=h(0,\pp{2}{0})\leq h(0,P_c)\leq P_c$, where in the first inequality we have used that $P_c$ is a global maximum for $h(0,\cdot)$. It follows that the whole orbit (except maybe for the initial value) of $\pp{2}{0}$ is contained in $(0,P_c]$, where the function is increasing. From the definition of $P_f$ and the monotonicity of the function we have
	\[0<\pp{2}{0}<P_f\,\Longrightarrow\,\pp{2}{0}<h(0,\pp{2}{0})<h(0,P_f)=P_f\]
	and hence for $0<\pp{2}{0}<P_f$ the sequence $\pp{2}{k}$ converges to $P_f$. Similarly, for $P_f<\pp{2}{0}<P_c$, the sequence $\pp{2}{k}$ decreases towards $P_f$, and hence we conclude that for any $\pp{2}{0}\in(0,1)$ the sequence converges to $P_f$, so
	\begin{equation}\label{coexfinal1}\etac\,=\,\psi(\bet{2}P_f).\end{equation}	
	\item Suppose now that $P_c<P_f$, which is equivalent to $h(0,P_c)>P_c$. Let $P_m=h(0,P_c)$ and observe that since $P_c$ is a global maximum for $h(0,\cdot)$, the orbit $\pp{2}{k}$ is contained in $[0,P_m]$ for any $\pp{2}{0}$. To control $\etac$ in this scenario observe that $h(0,\cdot)$ is decreasing in $[P_f,P_m]$ so
	\[P_f\leq \pp{2}{0}\leq P_m\,\Longrightarrow\,h(0,\pp{2}{0})\leq h(0,P_f)=P_f\]
	meaning that at least half of the points in the orbit of $\pp{2}{0}$ lie within $[0,P_f]$. Using that $\psi(x)$ is decreasing together with the previous observation, we conclude that
		\begin{equation}\label{coexfinal2}\etac\,\geq\,\sqrt{\psi(\bet{2}P_f)\,\psi(\bet{2}P_m)}\end{equation}
\end{itemize}

Finally, define $x_0$ as the only critical point of $\gg{2}$, which depends only on $\alp{2}$, and observe that using \eqref{PcPf} the condition $h(0,P_c)\leq P_c$ is equivalent to
$\pphi_2\tts x_0(x_0+\tfrac{1}{2})^3+\log(1-x_0)\leq 0$.
Solving for $\pphi{2}$ we obtain a condition of the form $\pphi{2}< z(\alp{2})\colonequals \frac{-\log(1-x_0)}{x_0(x_0+1/2)^3}$ and hence letting 
\begin{equation}\label{eq:defcf1}
\mathcal{F}_1(\alp{2},\pphi{2})\,=\,\left\{\begin{array}{cl}\frac{1}{\psi(\bet{2}P_f)}&\text{ if }\pphi{2}\leq z(\alp{2})\\[10pt]\frac{1}{\sqrt{\psi(\bet{2}P_f)\psi(\bet{2}P_m)}}&\text{ if }\pphi{2}> z(\alp{2})\end{array}\right.
\end{equation}
we get $\etac\pphi{1}>1$ if $\pphi{1}>\mathcal{F}_1(\alp{2},\pphi{2})$.
This finishes the proof of \Cref{teodinamico}(ii). In the next proposition we recap the properties of $\mathcal{F}_1$ that were stated in \Cref{teodinamico}, and whose proof we defer to the appendix.
\begin{prop}\label{propf1}
	Let $\mathcal{F}_1$ be as in \eqref{eq:defcf1}. Then, for fixed $\alpha$, $\mathcal{F}_1(\alpha,\pphi)$ is increasing as a function of $\pphi$ and satisfies $\mathcal{F}_1(\alpha,\pphi)=\Theta(\sqrt{\pphi\log(\pphi)})$ for large $\pphi$. In particular, for large $\pphi$ we have $\mathcal{F}_1(\alpha,\pphi)<\pphi$.
\end{prop}

\subsection{Proof of Theorem \ref{extcoex}(i)}

Our goal here is to prove that under the general assumption $\pphi{1}<\pphi{2}$ there are some $\bar{c},\varepsilon>0$ such that for any $0<c<\bar{c}$ the set $[0,1]\times[c,\pb{1}+\varepsilon]$ is interior-recurrent.
If $\pphi{2}>2\log2$ we are in the setting of \Cref{propc}(2), and taking $\bar{c}$ as in that statement yields the result (since we can extend it to all $\pp{2}\in(c,\pb{2}+\varepsilon_2)$ by continuity, provided $\varepsilon$ is sufficiently small).
Suppose then that $\pphi{2}\leq2\log2$.
One can check that $x_0$, which lives in $[0,1/2]$, is decreasing as a function of $\alpha$, while the function $x_0\longmapsto\frac{-\log(1-x_0)}{x_0(x_0+1/2)^3}$ is decreasing for $x_0\in[0,1/2]$, so $z(\alpha)$ is increasing with $z(0)=2\log2$.
In particular, we get
\[\pphi{1}<\pphi{2}\leq 2\log2\leq z(\alp{1})\]
and hence
\[\mathcal{F}_1(\alp{1},\pphi{1})=\frac{1}{\psi(\bet{1}P_f)}=\pphi{1}G_{\alp{1}}(1-e^{-\bet{1}P_f})\leq\pphi{1}<\pphi{2}\]
(here $P_f$ is the fixed point from $\hh{}{1}(P_f,0)=P_f$). Hence the same argument as the one used for coexistence (with reversed indexes) can be used to conclude that there are $\bar{c}$ and $\varepsilon$ small such that $[0,1]\times[c,\pb{1}+\varepsilon]$ is interior-recurrent for any $0<c<\bar{c}$.

\subsection{Proof of Theorem \ref{extcoex}(iii)}
\label{ext}

Our goal here is to prove that there is an interior-recurrent set $B$ where the stronger species survives while the density of the weaker one decays exponentially. We begin by fixing $\bar{c}$, $\varepsilon$, $\varepsilon'$ and $\ke$ as in \Cref{propc}.
Using these parameters we introduce an auxiliary set $B_1$, which we will refine until obtaining the desired set $B$, as
\[B_1\,=\,\Big\{p\in[0,\ke]\times[c,\pb{2}],\;\ll{1}(p)<1\Big\},\]
where $0<c<\bar{c}$ is a small parameter to be fixed later and $\pb{2}$ is as in \eqref{barp}. Recalling that $\ll{1}(p)=\hh{}{1}(p)/\pp{1}{}$, it follows that $B_1$ corresponds to a set of points whose first coordinate decreases after one iteration of $h$.
The cornerstone of this section is the following result:

\begin{lemma}
	\label{lemmaA}
	Let $a_1(x)$ be the solution of $a_1(x)=x(1-e^{-a_1(x)})$ and assume that $\pphi{1}$ and $\pphi{2}$ satisfy
	\begin{equation}
	\label{cond1}a_1(\pphi{1})\;<\;\frac{\pphi{2}}{1-\alp{2}}\min\tsm\Big\{g_{\alp{2}}(1-e^{-\frac{\pphi{2}}{2}}),g_{\alp{2}}(1-e^{-a_1(\pphi{1})})\Big\}.
	\end{equation} Then
	\begin{equation}\label{supinf}
	\sup\nolimits_{p\in B_1}\ll{1}\circ \hh(p)<1\quad\text{and}\quad\inf\nolimits_{\ll{1}(p)\geq 1}\ll{2}(p)>1.
	\end{equation}
\end{lemma}
In words, the first statement of \eqref{supinf} implies that when starting from $B_1$, after one iteration of the dynamical system the key feature $\ll{1}(p)<1$ is preserved, while the second one says that whenever the first coordinate increases, i.e. $\ll{1}(p)\geq 1$, the second component of $p$ increases by a constant factor, which will be used to show that the system eventually reaches $B_1$.

\begin{proof}
	We begin by observing that, under \eqref{cond1}, $\pphi{1}<\betac$. To see this, since $\pphi{1}<\pphi{2}$ we only need to worry about the case $\pphi{2}>\betac$, where condition \eqref{cond1} gives
	\[a_1(\pphi{1})\;<\;\tfrac{\pphi{2}}{1-\alp{2}}g_{\alp{2}}(1-e^{-\frac{\pphi{2}}{2}})\;\leq\;8\tts\pphi{2}(1-e^{-\frac{\pphi{2}}{2}})e^{-\frac{3\pphi{2}}{2}},\]
	where we have used that $\GG{2}(x)\leq 2(1-x)$; the function on the right hand side is decreasing in $(\betac,\infty)$, so $a_1(\pphi{1})\leq16\log 2\ts(1-e^{-\frac{2\log 2}{2}})e^{-3\log 2}=\log2$, and thus $\pphi{1}<2\log 2$, using the definition and monotonicity of $a_1(x)$. 
  Thanks to this bound on $\pphi{1}$, \Cref{crecimiento} states that $\ll{1}$ is decreasing in both $\pp{1}$ and $\pp{2}$, while $\hh{}{1}$ is increasing in $\pp{1}$ and decreasing in $\pp{2}$.
	
	From the monotonicity of $\ll{1}$ we deduce that the level set $\{\ll{1}(p)=1\}$ defines a strictly decreasing function $\pp{2}=s(\pp{1})$, for which there are values $a$ and $b$ such that $\ll{1}(a,c)=\ll{1}(0,b)=1$, where $c$ is as in the definition of $B_1$. Using these values we can easily characterize $B_1$ as a set bounded by the curves\\[-10pt]
	\begin{minipage}{.5\textwidth}
		\begin{center}
			\begin{align*}
			\mathcal{C}_1&\colonequals \left\{(\pp{1},c),\;a\leq \pp{1}\leq \ke\right\}\\
			\mathcal{C}_2&\colonequals \left\{(\ke,\pp{2}),\;c\leq \pp{2}\leq \pb{2}\right\}\\
			\mathcal{C}_3&\colonequals \left\{(\pp{1},\pb{2}),\;0\leq \pp{1}\leq \ke\right\}\\
			\mathcal{C}_4&\colonequals \left\{(0,\pp{2}),\;b\leq \pp{2}\leq \pb{2}\right\}\\
			\mathcal{C}_5&\colonequals \left\{(\pp{1},s(\pp{1})),\;0\leq \pp{1}\leq a\right\}
			\end{align*}
		\end{center}
	\end{minipage}%
	\begin{minipage}{.5\textwidth}
		\begin{center}
			\begin{tikzpicture}[scale=0.8]
			\draw[->] (-0.3,0) -- (3.8,0) node[right] {$p_1$};
			\draw[->] (0,-0.3) -- (0,3.8) node[above] {$p_2$};
			\draw[line width = 2, blue] (3.5,0.4) -- (3.5,3.5);
			\draw[line width = 2, blue!20!red] (1.2,0.4) -- (3.5,0.4);
			\draw[line width = 2, blue] (0,3.5) -- (3.5,3.5);
			\draw[line width = 2, blue!20!red] (0,3.5) -- (0,1.6);
			\draw[domain=0:1.2,smooth,variable=\x, line width = 2, red!80!yellow]  plot ({\x},{\x*\x/1.2-2*\x+1.6});
			\node[blue!20!red] at (-0.3,2.5) {$\mathcal{C}_4$};
			\node[blue!20!red] at (2.5,0.7) {$\mathcal{C}_1$};
			\node[blue] at (1.8,3.8) {$\mathcal{C}_3$};
			\node[blue] at (3.85,2.2) {$\mathcal{C}_2$};
			\node at (1.9,2.5) {$B_1$};
			\node[red!80!yellow] at (0.4,0.5) {$\mathcal{C}_5$};
			\draw (0,1.6) -- (-0.2,1.6) node[left] {$b$};
			\draw (0,0.4) -- (-0.2,0.4) node[left] {$c$};
			\draw (1.2,0) -- (1.2,-0.2) node[below] {$a$};
			\draw (3.5,0) -- (3.5,-0.2) node[below] {$\ke$};
			\end{tikzpicture}
		\end{center}
	\end{minipage}\\[3pt]	
	We will make use of the following lemma, whose proof we postpone.

\begin{lemma}
	\label{bordes}
	\begin{equation}\sup\nolimits_{p\in B_1}\ll{1}\circ\hh(p)\quad = \quad \max\nolimits_{p\in \mathcal{C}_1\cup\mathcal{C}_4\cup\mathcal{C}_5}\ll{1}\circ\hh(p).
	\end{equation}
\end{lemma}

Thus in order to prove the first statement in \eqref{supinf} we need to bound the maximum of $\ll{1}\circ\hh$ on each set $\mathcal{C}_1,\;\mathcal{C}_4$ and $\mathcal{C}_5$ separately.

Consider first $\mathcal{C}_1$, where $p=(\pp{1},c)$ with $\pp{1}\in[a,\ke]$. 
From \Cref{crecimiento} we know that $\ll{1}(\cdot,0)$ is strictly decreasing and, since $\pphi{1}<\betac$, the same proposition states that $\hh{}{1}(\cdot,0)$ is strictly increasing.
Therefore, since $h(\pp{1},0)=(\hh{}{1}(\pp{1},0),0)$ we deduce that the mapping $\pp{1}\longmapsto\ll{1}(\hh{}{1}(\pp{1},0),0)$ is strictly decreasing with its derivative bounded away from zero. Since all the functions involved in the argument are smooth, if $c$ is sufficiently small we also get that $\frac{\partial}{\partial \pp{1}} \ll{1}\circ\hh$ is negative and bounded away from zero on $\mathcal{C}_1$, so $\ll{1}\circ h$ is maximized at the point $(a,c)$, and we need to prove that it is smaller than 1 there. Indeed, using the definition of $a$ we obtain $\hh{}{1}(a,c)=a$, and since $a<\ke$ we can use \Cref{propc} to deduce that $\hh{}{2}(a,c)>c$, so we deduce that $h(a,c)\geq(a,c)$ (with strict inequality in the second component). Using this inequality and the monotonicity of $\ll{1}$ we finally conclude that $\ll{1}\circ\hh(a,c)<\ll{1}(a,c)=1$.

Next consider $\mathcal{C}_4$. Here we have $\pp{1}=0$, which greatly simplifies the analysis since
\[\hh{}{1}(0,\pp{2})=0,\qquad\hh{}{2}(0,\pp{2})=\gg{2}(1-e^{-\bet{2}\pp{2}}),\qquad\ll{1}\circ\hh=\pphi{1}\tfrac{1-e^{-\bet{2} \hh{}{2}}}{\bet{2} \hh{}{2}},\]
where $\hh{}{2}=\hh{}{2}(0,\pp{2})$. Indeed, from the particular form of $\ll{1}\circ\hh$ we have
\[\ll{1}\circ\hh<1\,\Longleftrightarrow\,\tfrac{1-e^{-\bet{2} \hh{}{2}}}{\bet{2} \hh{}{2}}<\tfrac{1-e^{-a_1(\pphi{1})}}{a_1(\pphi{1})}\] where we have used the definition of $a_1(\pphi{1})$ on the right hand side. Now, since the function $\tfrac{1-e^{-x}}{x}$
is decreasing we obtain
\begin{equation}\label{equivalente}\ll{1}\circ\hh(0,\pp{2})<1\quad\Longleftrightarrow\quad a_1(\pphi{1})<\bet{2}\hh{}{2}\,=\,\bet{2}\gg{2}(1-e^{-\bet{2}\pp{2}}). \end{equation}
Observe now that $\ll{1}$ is decreasing, so it is maximized at the points where $\hh{}{2}$ attains its minimum. Since $\gg{2}$ has a single local maximum it follows that $\hh{}{2}(0,\cdot)$ is minimized either where $\pp{2}$ is maximal or minimal. From this we conclude that the maximum of $\ll{1}$ on $\mathcal{C}_4$ is either $\ll{1}\circ\hh(0,\pb{2})$ or $\ll{1}\circ\hh(0,b)$. Now from \eqref{equivalente} we see that for $\ll{1}\circ\hh(0,\pb{2})<1$ to hold it is enough that \[a_1(\pphi{1})\,<\,\bet{2}\gg{2}(1-e^{-\bet{2}\pb{2}})\,\leq\,\bet{2}\gg{2}(1-e^{-\bet{2}\pphi{2}/2}),\]
which follows from our assumption \eqref{cond1}. To deal with $\ll{1}(0,b)$ we observe that $a_1(\pphi{1})=\bet{2}b$, so \eqref{equivalente} shows that $\ll{1}\circ\hh(0,b)<1$ if and only if $a_1(\pphi{1})<\bet{2}\gg{2}(1-e^{-a_1(\pphi{1})})$, which follows directly from \eqref{cond1}.

\smallskip

Finally, for $\mathcal{C}_5$, where $\ll{1}(\pp{1},\pp{2})=1$, it will be enough to show that 
\begin{equation}
\label{eqc5}
\inf\nolimits_{{p}\in\mathcal{C}_5}\big[\pphi{2}\tts\GG{2}^3\circ\ff{2}-\pphi{1}\tts\GG{1}^3\circ\ff{1}\big](p)>0.
\end{equation}
Indeed, if \eqref{eqc5} is satisfied then multiplying the inequality by $\frac{1-e^{-\Sigma_{p}}}{\Sigma_{p}}$, with $\Sigma_p=\bet{1}\pp{1}+\bet{2}\pp{2}$, gives $\ll{2}(p)>\ll{1}(p)=1$, and this implies $\pp{2}<\hh{}{2}(p)$, which in turn implies $\ll{1}(\hh)=\ll{1}(\pp{1},\hh{}{2})<\ll{1}(p)=1$.
To prove \eqref{eqc5} recall that $s(\pp{1})$ is a decreasing function, which means that $\ff{1}(\pp{1},s(\pp{1}))$ is increasing and $\ff{2}(\pp{1},s(\pp{1}))$ is decreasing. It follows that on $\mathcal{C}_5$ the function in \eqref{eqc5} is increasing on $\pp{1}$, so the infimum is positive if the inequality holds at $(0,b)$, which in this case follows from assumption \eqref{cond1}.

\smallskip

To complete the proof we need to show that $\inf_{p:\,\ll{1}(p)\geq1}\ll{2}(p)>1$, but $\ll{2}$ is decreasing in $\pp{2}$ and the maximal values of $\pp{2}$ within the region given by $\ll{1}\leq 1$ are found at $\ll{1}=1$. This way, it is enough to show that $\inf_{\ll{1}(p)=1}\ll{2}(p)>1$, and this is analogous to the proof of \eqref{eqc5}.
\end{proof}

It remains to prove \Cref{bordes}, which follows from similar monotonicity arguments.

\begin{proof}[Proof of \Cref{bordes}]
	Observe that, since $\ff{2}$ is increasing in $\pp{2}$ and decreasing in $\pp{1}$, the level sets $\{\ff{2}(p)=\gamma\}$ define strictly increasing functions $\pp{2}=r_{\gamma}(\pp{1})$. On these level sets $\hh{}{2}$ is clearly constant and $\hh{}{1}$ is increasing in $\pp{1}$; this last statement follows from the monotonicity of $\gg{1}$ (proved in \Cref{maxg}) and from
	$\ff{1}(\pp{1},r_{\gamma}(\pp{1}))+\gamma=(\ff{1}+\ff{2})(\pp{1},r_{\gamma}(\pp{1}))=1-\exp(-\bet{1}\pp{1}-\bet{2}r_{\gamma(\pp{1})})$,
	which implies that $\ff{1}$ increases in $\pp{1}$. Since $\ll{1}$ is decreasing in both arguments, at each level set $\ll{1}(h)$ attains its maximum at points of minimal values of $\pp{1}$. Our claim then is a result of the fact that each point $p\in A$ belongs to a level set $\ff{2}\equiv\gamma$ which attains a minimal value of $\pp{1}$ at $\mathcal{C}_1\cup\mathcal{C}_4\cup\mathcal{C}_5$.
\end{proof}

Observe that the condition $a_1(\pphi{1})<\bet{2}g_{\alp{2}}(1-e^{-a_1(\pphi{1})})$ appearing in \eqref{cond1} is equivalent to $\frac{\pphi{1}}{G_{\alp{2}}^3(1-e^{-a_1(\pphi{1})})}<\pphi{2}$. The left hand side of this inequality defines an increasing function of $\pphi{1}$ and $\alp{2}$, from \Cref{maxg} and the fact that $a_1(\pphi{1})$ is increasing with $\pphi{1}$, so the last inequality is equivalent to $\pphi{1}<\mathcal{F}_{2.1}(\alp{2},\pphi{2})$ for some implicit increasing function $\mathcal{F}_{2.1}$. Similarly, the condition $a_1(\pphi{1})<\bet{2}g_{\alp{2}}(1-e^{-\frac{\pphi{2}}{2}})$ appearing in \eqref{cond1} is equivalent to $\pphi{1}<\mathcal{F}_{2.2}(\alp{2},\pphi{2})$ for some $\mathcal{F}_{2.2}$. 
We then define the function $\mathcal{F}_2$ appearing in the statement of \Cref{teodinamico} as
\begin{equation}\label{eq:defcf2}
\mathcal{F}_{2}(\alp{2},\pphi{2})=\min\{\mathcal{F}_{2.1}(\alp{2},\pphi{2}),\mathcal{F}_{2.2}(\alp{2},\pphi{2})\}.
\end{equation}

	The rest of the proof of \Cref{extcoex}(ii) consists in modifying $B_1$ until obtaining the interior-recurrent set $B$ required in the result.
  As a first step, observe that from \Cref{lemmaA} there is some $\gamma\in(0,1)$ such that $\sup_{p\in B_1}\ll{1}\circ\hh(p)=\gamma$. We will build an interior-recurrent set $B_2$ by modifying slightly the definition of $B_1$. Define
	\[B_2=\big\{p\in[0,\ke]\times[c,\pb{2}],\;\ll{1}(p)<\bar{\gamma}\big\}\]
	for some $\bar{\gamma}\in(\gamma,1)$. We claim that this set is interior-recurrent with parameter $\bar{k}=1$. Indeed, for any $\pp\in B_2$, from our choice of parameters we have:
	\begin{itemize}[leftmargin=15pt]
		\item From \Cref{propc}.\ref{propc3} we have $\hh{}{1}(p)\leq(1-\varepsilon')\ke$.
		\item Since $\pp{1}\leq \ke$, from \Cref{propc}.\ref{propc1} we have $\hh{}{2}(p)\geq(1+\varepsilon')c$.
		\item From \Cref{lemmaA} we have $\sup_{p\in B_2}\ll{1}\circ\hh(p)\leq \sup_{p\in B_1}\ll{1}\circ\hh(p)=\gamma<\bar{\gamma}$.
	\end{itemize}
Hence there is some $\delta>0$ such that $d(\hh(\pp),B_2^c)>\delta$ uniformly on $\pp\in B_2$, which proves the claim. Now that we have shown that $B_2$ is interior-recurrent, we would like to show that there are $\gamma_1$ and $\gamma_2$ such that for any $\pp{}{}\in B_2$,
\[(1-\alp{1})\ff{1}(\pp)\leq \gamma_1\pp{1}\quad\text{and}\quad\gamma_2<\pp{2}.\]
Taking $\gamma_2=c$ the second inequality is trivially satisfied. The main problem is that in $B_2$ the decay we get is of the form $\hh{}{1}(\pp)\leq \bar{\gamma}\pp{1}$, which is not as strong as the one we need. However, once inside $B_2$ we have $\pp{1}{k}\longrightarrow0$, so in particular it is easy to see that for each $\delta$, the set $B_{\delta}\subseteq B_2$ given by
\[B_\delta:=\big\{p\in[0,\delta]\times[c,\pb{2}],\;\ll{1}(p)<\bar{\gamma}\big\}\]
is also interior-recurrent and satisfies the desired property. Indeed, for any $\varepsilon'>0$ we can take $\delta$ sufficiently small, so that for any $\pp{1}<\delta$ we have $\GG{1}^3\circ\ff{1}(p)\geq 1-\varepsilon'$.
Choosing $\varepsilon'$ sufficiently small, we use the inequality above to conclude that $(1-\alp{1})\ff{1}(p)\leq \frac{\bar{\gamma}}{1-\varepsilon'}\pp{1}$, and the result then follows taking $\gamma_1=\frac{\bar{\gamma}}{1-\varepsilon'}$.

It only remains to show that the dynamical system reaches $B_\delta$ in a bounded number of steps.
But, as claimed before, within $B_1$ we have $\sup_{p\in B_1}\ll{1}\circ\hh(p)=\gamma$ and hence the dynamical system reaches $B_{\delta}$ before $\log_{\gamma}(\delta)$ iterations. Thus it suffices to show that $\DS(h)$ reaches $B_1$ before $\bar{k}$ iterations for some fixed $\bar{k}\in\mathbb{N}$.
Fix an initial condition $\pp{}{0}$. If $\pp{1}{0}>\ke$, then by \ref{propc3} in \Cref{propc} we have $\pp{1}{1}\leq (1-\varepsilon')\pp{1}{0}$, and we may repeat the argument until the trajectory reaches $[0,\ke]\times[0,\pb{2}]$, where it remains forever. Since this procedure takes at most $\log_{1-\varepsilon'}(\kappa_\varepsilon)$ iterations, we may assume $\pp{1}{0}\leq\ke$.
Assume now that $l_2<\pp{2}{0}\leq c$ so we can use \ref{propc1} in \Cref{propc} to obtain $\pp{2}{1}>\pp{2}{0}(1+\varepsilon')$, and then repeat the argument to show that the sequence reaches $[0,\ke]\times[c,\pb{2}]$ in at most $\log_{1+\varepsilon'}(c/l_2)$ steps, remaining there forever. Hence we may assume that the initial condition $\pp{}{0}$ lies within this last set, and all we need to do is show that there is some bounded $n$ such that $\ll{1}(\pp{}{n})<1$. To do so observe from \Cref{lemmaA} that there is some fixed $\varepsilon>0$ such that for any $\pp{}{n}$ with $\ll{1}(\pp{}{n})\geq 1$, we necessarily have $\ll{2}(\pp{n})>1+\varepsilon$. It follows that if $\ll{1}(\pp{}{n})\geq 1$ for the first $n_0=\log_{1+\varepsilon}(1/c)$ iterations of the dynamical system, then $\pp{2}{n_0+1}>1$, which is impossible. We conclude that there must be some $n<n_0$ with $\ll{1}(\pp{}{n})< 1$ and hence the dynamical system reaches $B_1$ in a bounded number of iterations.

Finally, and as in the proof of \Cref{extcoex}(iii), we recap the properties of $\mathcal{F}_2$ that were stated in \Cref{teodinamico}, in the following proposition, whose proof we defer to the appendix.
\begin{prop}\label{propf2}
		Let $\mathcal{F}_2$ be as in \eqref{eq:defcf2}. Then $\mathcal{F}_2(\alpha,\pphi)>1$ and for fixed $\alpha$, $\mathcal{F}_2(\alpha,\pphi)= 1+(1+o(1))\frac{\pphi}{2}e^{-\frac{3\pphi}{2}}$ for large $\pphi$. On the other hand, for fixed $\pphi$, $\mathcal{F}_2(\alpha,\pphi)$ is decreasing as a function of $\alpha$.
\end{prop}

\appendix
\section{Technical proofs}
\label{apendice}

\begin{proof}[Proof of \Cref{convarbol}]
	Assume that \eqref{eq:alphaconv} holds and recall that $L_N= \log_2(N)/5$.
  Since $Z_0$ is a Bernoulli random variable with parameter $q$, we clearly have (with the obvious notation)
	\begin{align*}
	\E\big(Z_0(1-\alpha)^{Z_0+\cdots+Z_{L_N}}\big)&=\;q(1-\alpha)\E\big((1-\alpha)^{Z_1+\cdots+Z_{L_N}}\big)\\
	&=\;q(1-\alpha)\E\Big((1-\alpha)^{Z_1}\big(\E_1\big(1-\alpha)^{Z_2+\cdots+Z_{L_N}}\big)^{Z_1}\Big)\\
	&=\;q(1-\alpha)r((1-\alpha)W_{L_N}^2)
	\end{align*}
	where $r(x)=(qx+1-q)^3$ is the probability generating function of a Binomial$[3,q]$ random variable and for $k\geq2$ we let
  \[W_{k}=\E_1\big((1-\alpha)^{Z_2+\cdots+Z_{k}}\big)^{1/2},\]
  with $\E_1$ standing for the law of the Galton-Watson process with $Z_1=1$.
  To obtain an expression for $W_{L_N}$ we study the sequence $(W_k)_{k\geq 2}$ which, using the same reasoning as above, satisfies the quadratic recurrence equation
	\begin{equation}
	\label{eqrec}
	W_{k+1}\;=\;q(1-\alpha)W_{k}^2+1-q
	\end{equation}
	with initial condition $W_2=(1-\alpha)q+1-q$. This recurrence equation has two fixed points, 	
	$\frac{1\pm\sqrt{1-4q(1-q)(1-\alpha)}}{2q(1-\alpha)}$; the one with a plus is repulsive and larger than one while the one with a minus is attractive, so all orbits starting in $[0,1]$ converge to the latter, which we call $\We$.
  We then have
	$r((1-\alpha)\We^2)=\left[q(1-\alpha)\We^2+1-q\right]^3=\We^3$, and observing that $g_{\alpha}(q)=q(1-\alpha)\tts\We^3$, we deduce that \eqref{princaprox} is equivalent to
	\begin{equation*}
	\label{bound2}
	q(1-\alpha)\,\Big|r((1-\alpha)W_{L_N}^2)-r((1-\alpha)\We^2)\Big|\leq Ce^{-\alpha L_N}.
	\end{equation*}
	And since $q(1-\alpha)\leq 1$ and $|r(a)-r(b)|\leq 3|a-b|$ for all $a,b\in[0,1]$, it is enough to show that $|W_{L_N}-\We|\leq Ce^{-\alpha L_N}$. To this end we notice that, from the definition of $\We$,
	\begin{equation}
	\label{eq:2am}
	\begin{aligned}
	\big|W_{k+1}-\We\big|&=\;\Big|\big[q(1-\alpha)W_{k}^2+1-q\big]-\big[q(1-\alpha)\We^2+1-q\big]\Big|\\
	&=\;q(1-\alpha)\big|W_k-\We\big|\big(W_k+\We\big)\leq\;q(1-\alpha)\big|W_k-\We\big|\big(1+\We\big),
	\end{aligned}
	\end{equation}
	but it can be easily shown that $q(1+\We)\leq 1$, so $\big|W_{k+1}-\We\big|\leq  (1-\alpha)\big|W_{k}-\We\big|$ for all $k\geq 2$. In particular we get
	\begin{equation}
	\label{eq:1am}
	|W_{L_N}-\We|\leq 2(1-\alpha)^{L_N-1}\leq Ce^{-\alpha L_N}
	\end{equation}
	where $C>0$ is independent of $q$, and $\alpha$. 
\end{proof}

\begin{remark}\label{rem:lalala}
Assume that $\alp{}{N}$ is a sequence in $[0,1]$ such that $\alp{}{N}\to0$ and
\begin{align}\label{cond:3}
\alp{}{N}\log(N)/\log(\log(N)) \longrightarrow\infty.
\end{align}
We will explain how to improve the bound of \Cref{convarbol} in this case. One consequence of this is that in \Cref{theo:5intro} all the factors $\ualpha_{N}\log(N)$ appearing in the exponents in \eqref{esperanzacoex}--\eqref{esperanzadom2} can be replaced by $\ualpha_{N}\log(N)\vee\log(N)^{1/2}$ . This follows by noting that all other bounds in the proof of \Cref{theo:5intro} are of smaller order.
	
\noindent Fix $N$ large and use \eqref{eq:1am} to bound the distance between the $\sfrac{L_N}{2}$-th term of the sequence and $\We$, leading to
\[|W_{L_N/2}-\We|\leq 2e^{-\frac{\alp{}{N} (L_N-2)}{2}}\leq Ce^{-2\log(\log(N))}=C(\log(N))^{-2}\leq C(\alpha_N)^2\]
for some $C$ independent of $q$, where in the second inequality we used \eqref{cond:3} and in the third one we used $\alpha_N\log(N)\rightarrow \infty$.
Noticing that $W_k$ converges monotonically to $\We$, the above bound is valid for all $W_k$ with $k\geq\sfrac{L_N}{2}$, so we can restart the sequence at the $\sfrac{L_N}{2}$-th term to improve the bound in \eqref{eq:2am} to
\begin{equation*}
\big|W_{k+1}-\We\big|=\;q(1-\alp{}{N})\big|W_k-\We\big|\big(W_k+\We\big)\leq\;q(1-\alp{}{N})\big|W_k-\We\big|\big(C(\alpha_N)^2+2\We\big).
\end{equation*}
But $2q(1-\alp{}{N})\We=1-\sqrt{1-4q(1-q)(1-\alp{}{N})}\leq 1-\sqrt{\alp{}{N}}$ so we have $\big|W_{k+1}-\We\big|\;\leq\;\big|W_k-\We\big|\big[1-\sqrt{\alp{}{N}}+C(\alpha_N)^2\big]$ for all $k\geq\sfrac{L_N}{2}$. In particular, since $\alp{}{N}\to0$,
\[|W_{L_N}-\We|\;\leq\;2\big[1-\sqrt{\alp{}{N}}+C(\alpha_N)^2\big]^{L_N/2}\leq Ce^{-\frac{\sqrt{\alp{}{N}}\log N}{2}}\;\leq\;Ce^{-\sqrt{\log N}},\]
where we used that $\alp{}{N}\log N\to\infty$ as $N\to\infty$.
\end{remark}

We turn now to the remaining proofs from Section \ref{preliminaries}.

\begin{proof}[Proof of \Cref{maxg}]
	We prove only the case $\alp>0$; the case $\alp=0$ is similar but much easier to handle. Observe first that $\GG{}(x)$ satisfies
	\begin{equation}
	\label{eq1ap}
	\GG{}(x)\sqrt{1-4(1-\alp)x(1-x)}=-\GG{}(x)+2-2x,
	\end{equation}
	\begin{equation}
	\label{eq2ap}
	\GG{}'(x)\;=\;\tfrac{\GG{}(x)-1}{x\sqrt{1-4(1-\alp)x(1-x)}}\;=\;\tfrac{\GG{}(x)-1}{x[1-2(1-\alp)x\GG{}(x)]}.
	\end{equation}
	To find the maximum of $\gg{}$ we solve the first order condition
	$0=\gg{}'(x)= x\tts\GG{}^3(x)\tsm\left[\frac{1}{x}+\frac{3\GG{}'(x)}{\GG{}(x)}\right]$.
	The factor $x\GG{}^3(x)$ equals $0$ only at $0$ and $1$, so $\gg{}'(x)=0$ inside $(0,1)$ only if the factor in brackets vanishes which, from the above identities, means that $\GG{}(x)=x+1/2$. We conclude, since $\GG{}\leq 1$, that every critical point of $g_\alpha$ must lie in $[0,1/2]$.	
	The first part of the proposition will follow if we show that at every such critical point $x_0$ we have $\gg{}''(x_0)<0$ (so every critical point is a maximum, and hence there can only be one). Now $\gg{}''(x_0)= \gg{}(x_0)\left[\frac{3\GG{}''(x_0)}{\GG{}(x_0)}-\frac{4}{3x_0^2}\right]$, so it suffices to prove that $\GG{}''(x_0)\leq0$. Using \eqref{eq1ap} and \eqref{eq2ap} we find	$\GG{}''(x)= \frac{[\GG{}(x)-1]2(1-\alp)x[2\GG{}(x)+x\GG{}'(x)]}{[x(1-2(1-\alp)x\GG{})]^2}$, which is non-positive as soon as $2\GG{}(x_0)+x\GG{}'(x_0)\geq0$ since $\GG{}\leq 1$. By \eqref{eq1ap} and \eqref{eq2ap} again, this is equivalent to $3-4x>\GG{}(x_0)$, which is satisfied because thanks to the condition $x_0\in[0,1/2]$.
	
	To prove the second part of the proposition write $\Sigma_p=\bet{1}\pp{1}+\bet{2}\pp{2}$ so that
	\[\ff{i}(p)\;=\;\tfrac{1-e^{-\Sigma_p}}{\Sigma_p}\bet{i}\pp{i}.\]
	Since $x\mapsto \frac{1-e^{-x}}{x}$ is decreasing, it follows that $\ff{i}(p)\leq1-e^{-\bet{i} \pp{i}}\leq1-e^{-\bet{i}\gg{i}\!(x_0)}$
	so it will be enough to prove that $1-e^{-\bet{i} \gg{i}\!(x_0)}< x_0$. Since $x_0$ is characterized by $\GG{i}(x_0)=x_0+1/2$, it is enough to show that	$V(x_0)\colonequals \pphi{i}x_0 \big(\tfrac{1}{2}+x_0\big)^3+\log(1-x_0)< 0$.
	But, in fact, $V$ is non-positive on the entire interval $(0,1/2]$. Indeed, $V(0)=0$ and $V(1/2)=\frac{\pphi{i}}{2}-\log 2$, which is negative from our assumption $\pphi{i}<\betac$, so it is enough to prove that the inequality holds at the critical points of $V$; this follows from	$V'(x)=\pphi{i}(\tfrac{1}{2}+x)^2(\tfrac{1}{2}+4x)-\tfrac{1}{1-x}$,	$V''(x)=\pphi{i}(\tfrac{1}{2}+x)(3+12x)-\tfrac{1}{(1-x)^2}$, so whenever $V'(x_1)=0$ we have $(1-x_1)V''(x_1)=\pphi{i}(x_1+1/2)[-16x_1^2+13 x_1/2+11/4]$, which is positive in $[0,1/2]$, giving that $x_1$ is a minimum.
\end{proof}

\begin{proof}[Proof of \Cref{crecimiento}]
	We keep the notation $\Sigma_p$ used in the previous proof.
  For the dependence of $\ff{1}$ on $\pp{1}$ we write the function as $(1-e^{-\Sigma_p})\tfrac{\bet{1}\pp{1}}{\Sigma_p}$ which, for fixed $\pp{2}$, is the product of two increasing functions. For the dependence of $\ff{1}$ on $\pp{2}$, on the other hand, we write $\ff{1}$ as $\frac{1-e^{-\Sigma_p}}{\Sigma_p}\bet{1}\pp{1}$; the factor on the left is decreasing in $\pp{2}$ while the one on the right is constant. This gives (1). Next observe that $\ll{1}(p)=\pphi{1}\frac{1-e^{-\Sigma_p}}{\Sigma_p}\GG{1}^3\circ\ff{1}(p)$ and the same analysis shows that $\ff{1}$ is increasing and $\GG{}$ is decreasing, giving (2).
		
  If $\pphi{1}<\betac$, then from \Cref{maxg} we know that $\gg{i}'\circ\ff{i}(\pp)\geq0$, so $\hh{}{1}$ satisfies the same monotonicity as $\ff{1}$ on each argument. Since $\ll{1}(p)=\tfrac{\hh{}{1}(p)}{\pp{1}}$, it must behave as $\hh{}{1}$ with respect to $\pp{2}$. This gives (3).
\end{proof}

\begin{proof}[Proof of \Cref{propc}] 
We keep again the definition of $\Sigma_p$ used in the proof of \Cref{maxg}. Let us show first that the equation
 	\[\gg{1}(1-e^{-\bet{1} \ke})\,=\,(1-\varepsilon)\ke\]
has indeed a unique positive solution. To see this define $y=1-e^{-\bet{1} \ke}$ and observe that $\ke$ is a positive solution of the above equation if and only if $y$ is a solution of
\[\pphi{1}\GG{1}^3(y)\,=\,\frac{-(1-\varepsilon)\log(1-y)}{y}.\]
However, $\GG{1}^3(y)$ is a decreasing function with $\GG{1}^3(0)=1$, while $-\frac{\log(1-y)}{y}$ is increasing and tends to $1$ as $y\to0$. Since $\pphi{1}>1>1-\varepsilon$, this implies that there is exactly one positive solution $y>0$. Furthermore, taking $\ke=1-\alp{1}$ we obtain
\[(1-e^{-\pphi{1}})\GG{1}^3(1-e^{-\pphi{1}})<1-\varepsilon\]
provided $\varepsilon$ is sufficiently small, since both terms on the left are smaller than $1$ for $\pphi{1}>1$. We thus deduce that $\ke<1$.
	To prove \eqref{propc11} we take $c$ small (to be fixed later) and suppose that $\pp{2}<c$. Observing that $\ff{2}(p)=\frac{1-e^{-\Sigma_p}}{\Sigma_p}\bet{2}\pp{2}$ we deduce that if $c$ small enough, $\frac{1-e^{-\bet{1} \pp{1}}}{(1-\varepsilon)\bet{1} \pp{1}}\bet{2}\pp{2}<\ff{2}(p)<\bet{2}\pp{2}$ for $\varepsilon$ small, so from the monotonicity of $\GG{}$, we obtain
	\begin{equation}\label{desfinal}\ll{2}(p)=(1-\alp{2})\tfrac{\ff{2}(p)}{\pp{2}} \GG{2}^3\circ \ff{2}(p)\geq \pphi{2} \tfrac{1-e^{-\bet{1} \pp{1}}}{(1-\varepsilon)\bet{1} \pp{1}}\GG{2}^3(\bet{2}c).\end{equation}
	Since the fraction is decreasing in $\pp{1}$ we obtain a lower bound by taking $\pp{1}=\ke$ and using its definition to obtain
	$\ll{2}\geq \frac{\pphi{2} }{\pphi{1}}\frac{\GG{2}^3(\bet{2}c)}{\GG{1}^3(1-e^{-\bet{1}\ke})}$.
  Recalling that $\tfrac{\pphi{2}}{\pphi{1}}>1$ we have $\GG{2}\leq 1$ and as $c\to0$ we have $\GG{2}(\bet{2}c)\to1$, so taking first $\varepsilon$ small and then $c$ sufficiently small, the right hand side is larger than $1+\varepsilon'$ for some $\varepsilon'$.
	
	For \eqref{propc12}, \Cref{maxg} gives that $\gg{2}$ has a single critical point which is a maximum, so $\hh{}{2}=\gg{2}\circ\ff{2}$ is minimized either when $\ff{2}$ is minimized or maximized. Remembering that $\ff{2}$ decreases with $\pp{1}$ and increases with $\pp{2}$, we conclude that the minimum of $\hh{}{2}$ over the set $[0,\ke]\times[c,\pb{2}]$ is obtained either at $(0,\tfrac{1-\alp{2}}{2})$ or at $(\ke,c)$. We already saw that at $p=(\ke,c)$ we have $\hh{}{2}(p)=\ll{2}(p)\pp{2}>(1+\varepsilon')c$, meaning that we need only to control $\hh{}{2}$ at $(0,\tfrac{1-\alp{2}}{2})$, where it equals $\gg{2}(1-e^{-\pphi{2}/2})$, so the result follows by taking $c$ small enough so that $\gg{2}(1-e^{-\pphi{2}/2})>(1+\varepsilon')c$.
	
	To get \ref{propc2} in the proposition we need to extend the above properties to general values of $\pp{1}$. 
	We proceed as before, but when computing \eqref{desfinal} we use the additional information $\pphi{2}>\betac$ to improve the lower bound without imposing any restriction on $\pp{1}$. Indeed, since $\pphi{1}<\pphi{2}$ we deduce
	that $\bet{1}\pp{1}\leq\tfrac{\pphi{2}}{2}$ so, from monotonicity of $\tfrac{1-e^{-x}}{x}$,
	\[\textstyle\ll{2}(p)\geq \pphi{2} \frac{1-e^{-\bet{1} \pp{1}}}{(1-\varepsilon)\bet{1} \pp{1}}\GG{2}^3(\bet{2}c)\geq 2(1-e^{-\phi_2/2})\GG{2}^3(\bet{2}c),\]
	but $2(1-e^{-\phi_2/2})>1$  from the assumption on $\pphi{2}$, so taking $c$ sufficiently small we conclude again that $\ll{2}(p)>1+\varepsilon'$ for some $\varepsilon'$ small. The proof of the second property is exactly the same as in \eqref{propc12}.
	
	We turn finally to \eqref{propc31} and \eqref{propc32}. Notice that, since $\pphi{1}<\betac$, from \Cref{crecimiento} we know that $\hh{}{1}$ is increasing in $\pp{1}$ and decreasing in $\pp{2}$, so using the definition of $\ke$ we deduce
	\[\pp{1}<\ke\;\Longrightarrow\;\hh{}{1}(p)\leq\hh{}{1}(\ke,0)=\gg{1}(1-e^{-\bet{1}\ke})=(1-\varepsilon)\ke,\]
	which proves \eqref{propc31}. To prove \eqref{propc32} we use a similar argument with $\ll{1}$, which we know is decreasing in both arguments, so that
	\[\ke<\pp{1}\;\Longrightarrow\;\ll{1}(p)\leq\ll{1}(\ke,0)=\tfrac{\gg{1}(1-e^{-\bet{1}\ke})}{\ke}=(1-\varepsilon),\]
	and the result follows.	
\end{proof}

\begin{proof}[Proof of \Cref{propf1}] 
Recall the definition \eqref{eq:defcf1} of $\mathcal{F}_1(\alp{2},\pphi{2})$:
\begin{equation*}
\mathcal{F}_1(\alp{2},\pphi{2})\,=\,\left\{\begin{array}{cl}\frac{1}{\psi(\bet{2}P_f)}&\text{ if }\pphi{2}\leq z(\alp{2})\\[5pt]\frac{1}{\sqrt{\psi(\bet{2}P_f)\psi(\bet{2}P_m)}}&\text{ if }\pphi{2}> z(\alp{2})\end{array}\right.,
\end{equation*}
where $\psi(x)=\frac{1-e^{-x}}{x}$ and where, taking $x_0$ as the only critical point of $\gg{2}$ (as seen in \Cref{maxg}), the values $z(\alp{2})$ and $P_m$ are defined as 
\[z(\alp{2})=\frac{-\log(1-x_0)}{x_0(\frac{1}{2}+x_0)^3},\qquad\text{and}\qquad P_m=g_{\alp{2}}(x_0)\]
(and hence do not depend on $\pphi{2}$) while $P_f$ is the only positive solution of $g_{\alp{2}}(1-e^{-\bet{2}P_f})=P_f$. To show that $\mathcal{F}_1(\alp{2},\pphi{2})$ is increasing as a function of $\pphi{2}$ define $x_f=1-e^{-\bet{2}P_f}$ which, by the definition of $P_f$, satisfies
\begin{equation}\label{app1}\phi_2=-\frac{\log(1-x_f)}{x_fG_{\alp{2}}^3(x_f)}.\end{equation}
But the function $x\to\frac{-\log(1-x)}{x}$ is strictly increasing, and 
$x\to G_{\alp{2}}(x)$ is strictly decreasing, so $x_f$ increases as a function of $\pphi{2}$. In particular, since
\[\frac{1}{\psi(\bet{2}P_f)}=\frac{-\log(1-x_f)}{x_f},\]
we deduce that up to $z(\alp{2})$ the function $\mathcal{F}_1(\alp{2},\cdot)$ is increasing. At $\pphi{2}=z(\alp{2})$ we have (by definition of $z(\alp{2})$)
\[\frac{-\log(1-x_f)}{x_fG_{\alp{2}}^3(x_f)}=\phi_2=\frac{-\log(1-x_0)}{x_0(\frac{1}{2}+x_0)^3}=\frac{-\log(1-x_0)}{x_0G_{\alp{2}}^3(x_0)}\]
where the last equality follows from $G_{\alp{2}}(x_0)=\frac{1}{2}+x_0$, which was shown in \Cref{maxg}. Since the function $x\to \frac{-\log(1-x)}{xG_{\alp{2}}^3(x)}$ is strictly increasing we deduce that $x_0=x_f$, but then
\[P_m=g_{\alp{2}}(x_0)=g_{\alp{2}}(x_f)=g_{\alp{2}}(1-e^{-\bet{2}P_f})=\hh(0,P_f)=P_f,\]
and hence
\[\frac{1}{\psi(\bet{2}P_f)}=\frac{1}{\sqrt{\psi(\bet{2}P_f)\psi(\bet{2}P_m)}},\]
so $\mathcal{F}_1(\alp{2},\pphi{2})$ is continuous at $z(\alp{2})$. It remains to show that for $\pphi{2}>z(\alp{2})$ the function is also increasing, but we already saw that $\frac{1}{\psi(\bet{2}P_f)}$ satisfies this property, so the function will be increasing as soon as $\frac{1}{\psi(\bet{2}P_m)}$ is increasing as well. Now, $P_m$ is independent of $\pphi{2}$ and $\psi$ is a decreasing function so $\frac{1}{\psi(\bet{2}P_m)}=\frac{1}{\psi(\pphi{2}\frac{P_m}{1-\alp{2}})}$ must be indeed increasing.

For the asymptotic analysis we  deduce from \eqref{app1} that $\lim_{\pphi{2}\to\infty}x_f(\pphi{2})=1$, and since $\lim_{x\to1}-\frac{\log(1-x)}{x}=1$ and $\lim_{x\to 1}\frac{G_{\alp{2}}(x)}{1-x}= 1$, taking sufficiently large $C$ and small $\varepsilon$ we have
\[C^{-1}\pphi{2}^{-1/3-\varepsilon}\leq 1-x_f\leq C\pphi{2}^{-1/3+\varepsilon}.\]
From this analysis we deduce that
\[\frac{1}{\psi(\bet{2}P_f)}=\frac{-\log(1-x_f)}{x_f}=\Theta(\log(\pphi{2}))\]
while for the factor $\frac{1}{\psi(\bet{2}P_m)}$ recall that $\frac{1}{\psi(\bet{2}P_m)}=\frac{\pphi{2}\frac{P_f}{1-\alp{2}}}{1-\exp(-\pphi{2}\frac{P_f}{1-\alp{2}})}=\Theta(\pphi{2})$ since $P_m$ does not depend on $\pphi{2}$. We deduce that
\[\frac{1}{\sqrt{\psi(\bet{2}P_f)\psi(\bet{2}P_m)}}\,=\,\Theta(\sqrt{\pphi{2}\log(\pphi{2})})\]
as claimed.
\end{proof}

\begin{proof}[Proof of \Cref{propf2}] 
	Recall that $\mathcal{F}_2(\alp{2},\pphi{2})$ was defined in \eqref{eq:defcf2} as
	\[\mathcal{F}_2(\alp{2},\pphi{2})=\min\{\mathcal{F}_{2.1}(\alp{2},\pphi{2}),\mathcal{F}_{2.2}(\alp{2},\pphi{2})\}\]
	where for $\alp{2}$ fixed:
	\begin{enumerate}\setlength\itemsep{4pt}
		\item $\mathcal{F}_{2.1}(\alp{2},\cdot)$ is the inverse function of $x\to\frac{x}{G_{\alp{2}}^3(1-e^{-a_1(x)})}$,
		\item $\mathcal{F}_{2.2}(\alp{2},\pphi{2})=a_1^{-1}\left(\frac{\pphi{2}}{1-\alp{2}}g_{\alp{2}}(1-e^{-\pphi{2}/2})\right)$
	\end{enumerate}
		and where $a_1(x)$ is defined for $x>1$ as the only positive solution of $a_1(x)=x(1-e^{-a_1(x)})$. We begin the proof by studying the asymptotic behavior of $\mathcal{F}_2(\alp{2},\pphi{2})$. Observe that as $\pphi{2}\to\infty$ we have
		\[\frac{\pphi{2}}{1-\alp{2}}g_{\alp{2}}(1-e^{-\pphi{2}/2})\,=\,(1+o(1))\pphi{2}e^{-3\pphi{2}/2}\]
		where the term $\pphi{2}e^{-3\pphi{2}/2}$ converges to zero as $\pphi{2}\to\infty$. It follows from the definition of $a_1$ that
		\begin{equation}\label{app2}\mathcal{F}_{2.2}(\alp{2},\pphi{2})\,=\,\frac{(1+o(1))\pphi{2}e^{-3\pphi{2}/2}}{1-e^{-(1+o(1))\pphi{2}e^{-3\pphi{2}/2}}}=1+(1+o(1))\frac{\pphi{2}}{2}e^{-3\pphi{2}/2},\end{equation}
		thus showing the asymptotic behavior of $\mathcal{F}_2$. To prove that $\mathcal{F}_2(\alp{2},\pphi{2})>1$ we must show that both $\mathcal{F}_{2.1}(\alp{2},\pphi{2})>1$ and $\mathcal{F}_{2.2}(\alp{2},\pphi{2})>1$. For the inequality involving $\mathcal{F}_{2.1}$ observe that \[\lim_{x\to\infty}\frac{x}{G_{\alp{2}}^3(1-e^{-a_1(x)})}=\infty\]
		and that $\mathcal{F}_{2.1}(\alp{2},\cdot)$ is continuous so the statement $\mathcal{F}_{2.1}(\alp{2},\pphi{2})>1$ fails if and only if we can find some $\pphi{2}>1$ such that $\mathcal{F}_{2.1}(\alp{2},\pphi{2})=1$. This equation implies that such a $\pphi{2}$ must satisfy
		\[\pphi{2}\,=\,\frac{1}{\GG{2}^3(1-e^{-a_1(1)})}\]
		where $a_1(1)$ is defined by continuity as $a_1(1)=\lim_{x\to1}a_1(x)=0$. It follows that the denominator is equal to $\GG{2}^3(0)=1$ and hence $\pphi{2}=1$, contradicting our hypothesis $\pphi{2}>1$ so we conclude that $\mathcal{F}_{2.1}(\alp{2},\pphi{2})>1$. The inequality $\mathcal{F}_{2.2}(\alp{2},\pphi{2})>1$ follows directly from \eqref{app2} since $\mathcal{F}_{2.2}$ is of the form $\frac{y}{1-e^{-y}}$ for some positive $y$. Finally, for $\pphi{2}$ fixed take $\alp{2}<\alp{2}'$ and notice that since
	\[\frac{\mathcal{F}_{2.1}(\alp{2},\pphi{2})}{G_{\alp{2}}^3(1-e^{-a_1(\mathcal{F}_{2.1}(\alp{2},\pphi{2}))})}=\pphi{2}\]
	and that $G_{(\cdot)}(x)$ is decreasing we deduce
	\[\frac{\mathcal{F}_{2.1}(\alp{2},\pphi{2})}{G_{\alp{2}'}^3(1-e^{-a_1(\mathcal{F}_{2.1}(\alp{2},\pphi{2}))})}>\pphi{2}=\frac{\mathcal{F}_{2.1}(\alp{2}',\pphi{2})}{G_{\alp{2}'}^3(1-e^{-a_1(\mathcal{F}_{2.1}(\alp{2}',\pphi{2}))})}\]
	and hence, from monotonicity we conclude $\mathcal{F}_{2.1}(\alp{2}',\pphi{2})<\mathcal{F}_{2.1}(\alp{2},\pphi{2})$. Similarly, observing that
	\[\mathcal{F}_{2.2}(\alp{2},\pphi{2})=a_1^{-1}\left(\pphi{2}(1-e^{-\pphi{2}/2})G^3_{\alp{2}}(1-e^{-\pphi{2}/2})\right)\]
	and that $a_1^{-1}$ is increasing and the argument is decreasing with $\alp{2}$ we conclude that the function $\mathcal{F}_{2.2}(\alp{2},\pphi{2})$ is decreasing on this parameter.
\end{proof}

\smallskip
\noindent\textbf{Acknowledgements:} 
The authors would like to thank the referees and editors for their constructive comments, which helped improve considerably this manuscript. LF also thanks J.F. Marckert for the assistance, discussions and comments that improved this article. 
This project began as part of LF's Master thesis at U. de Chile, and he acknowledges support from LaBRI and a CONICYT Master Scholarship. 
This project was also part of AL's Ph.D. thesis at U. de Chile, and he acknowledges support by the CONICYT-PCHA/Doctorado nacional/2014-21141160 scholarship.
DR was supported by Fondecyt Grants 1160174 and 1201914.
All three authors were also supported by Centro de Modelamiento Matemático (CMM) Basal Funds  ACE210010 and FB210005 from ANID-Chile, and by Programa Iniciativa Cient\'ifica Milenio grant number NC120062 through Nucleus Millenium Stochastic Models of Complex and Disordered Systems.

\printbibliography[heading=apa]

\end{document}